\documentclass{article}
\usepackage[left=3cm,right=3cm,top=3cm,bottom=3cm]{geometry} 
\usepackage{amsmath,amssymb} 
\usepackage{amsmath, amsthm,amssymb,amsfonts,latexsym}
\usepackage{amsthm}
\usepackage{tikz}
\usepackage{lipsum}
\usepackage{color}

\usepackage{hyperref}
\usepackage[capitalise]{cleveref}
\usetikzlibrary{patterns}
\usetikzlibrary{intersections}
\usetikzlibrary{arrows,positioning}
\usepackage{graphicx}
\usepackage{pgfplots}
 \pgfplotsset{compat=1.14}
\usetikzlibrary{plotmarks}
  \usepgfplotslibrary{patchplots}
  \usepackage{grffile}
\usepgfplotslibrary{fillbetween}
\usepackage{float}
\usepackage{fancyhdr}
\usepackage{caption}
\usepackage{subcaption}

\newcommand*{\rom}[1]{\expandafter\@slowromancap\romannumeral #1@}

\newcommand{\gettikzxy}[3]{%
  \tikz@scan@one@point\pgfutil@firstofone#1\relax
  \edef#2{\the\pgf@x}%
  \edef#3{\the\pgf@y}%
}

\def\namedlabel#1#2{\begingroup
	#2%
	\def\@currentlabel{#2}%
	\phantomsection\label{#1}\endgroup
}

\newcommand{\sA}{\mathcal A}
\newcommand{\sB}{\mathcal B}
\newcommand{\sC}{\mathcal P}
\newcommand{\sD}{\mathcal D}

\newcommand{\sI}{\mathcal I}

\newcommand{\sM}{\mathcal M}

\newcommand{\sP}{\mathcal P}

\newcommand{\R}{\mathbb R}
\newcommand{\E}{\mathbb E}

\newcommand{\arginf}{\mbox{arginf}}

\newtheorem{thm}{Theorem}
\newtheorem{prop}{Proposition}
\newtheorem{eg}{Example}
\newtheorem{lem}{Lemma}
\newtheorem{cor}{Corollary}
\newtheorem{rem}{Remark}

\newtheorem{defn}{Definition}

\begin{document}
\title{The potential of the shadow measure}
\author{
Mathias Beiglb\"{o}ck\thanks{Fakult\"{a}t f\"{u}r Mathematik, Universit\"{a}t Wien \textit{mathias.beiglboeck@univie.ac.at}}\hspace{10mm}
David Hobson\thanks{Department of Statistics, University of Warwick \textit{d.hobson@warwick.ac.uk}} \hspace{10mm}
Dominykas Norgilas\thanks{Department of Mathematics, University of Michigan \textit{dnorgila@umich.edu}}
}
\date{\today}
\maketitle

\begin{abstract}
It is well known that given two probability measures $\mu$ and $\nu$ on $\R$ in convex order there exists a discrete-time martingale with these marginals. Several solutions are known (for example from the literature on the Skorokhod embedding problem in Brownian motion). But, if we add a requirement that the martingale should minimise the expected value of some functional of its starting and finishing positions then the problem becomes more difficult. Beiglb\"{o}ck and Juillet (Ann. Probab. 44 (2016) 42–106) introduced the \emph{shadow} measure which induces a family of martingale couplings, and solves the optimal martingale transport problem for a class of bivariate objective functions. In this article we extend their (existence and uniqueness) results by providing an explicit construction of the shadow measure and, as an application, give a simple proof of its associativity. \\

\end{abstract}

\section{Introduction}
Given two probability measures $\mu,\nu$ on $\R$, a \textit{transport plan} or \textit{coupling} between $\mu$ and $\nu$ is a probability measure $\pi$ on $\R^2$ such that
$\pi(A\times\R)=\mu(A)$ and $\pi(\R\times B)=\nu(B)$
for all Borel sets $A,B$ of $\R$. It is often convenient to express a coupling $\pi$ via its disintegration with respect to the first marginal $\mu$,
$
\pi(dx,dy)=\mu(dx)\pi_x(dy)
$
where $x\mapsto\pi_x$ is a $\mu$-almost surely unique probability kernel. In the language of the classical optimal transport, each transport plan $\pi$ corresponds to a joint distribution of $X\sim\mu$ and $Y\sim\nu$ and then it is natural to ask for couplings which, (not only possess nice structures or properties but also) for a given cost function $c:\R\to\R$, minimise the expected cost $\E^\pi[c(X,Y)]$. Brenier's Theorem (see Brenier \cite{Brenier:87}, and R\"{u}schendorf and Rachev \cite{RuRach:90}) considers the problem in $\R^d$ with an  Euclidean cost $c(x,y)=\lvert y-x\lvert^2$. Then, under some regularity assumptions, the optimal coupling is a push forward measure induced by the gradient of a convex function $\phi:\R^d\mapsto\R$. In other words, the optimal coupling $\hat\pi$ is \textit{deterministic} and of the form
$$
\hat\pi(dx,dy)=\mu(dx)\delta_{\nabla\phi(x)}(dy).
$$
In one dimension, this says that $\hat\pi$ is concentrated on a graph of an increasing function. Furthermore, it is optimal for (at least) costs $c(x,y)=h(y-x)$, where $h:\R\mapsto\R$ is strictly convex, and coincides with the monotone (Hoeffding-Fr\'{e}chet) coupling $\pi^{HF}$ (which is often called the \textit{quantile} coupling).

In the martingale version of optimal transport, introduced in the context of a specific cost function by Hobson and Neuberger \cite{HobsonNeuberger:12} and more generally by Beiglb\"{o}ck et al. \cite{BeiglbockHenryLaborderePenkner:13} and Galichon et al. \cite{GalichonHenryLabordereTouzi:14}, the goal is to construct a martingale $(M_i)_{i=1,2}$ with $M_1\sim\mu$ and $M_2\sim\nu$, and such that $\E[c(M_1,M_2)]$ is minimised. The martingale requirement places a non-trivial constraint on the possible joint distributions of $M_1$ and $M_2$. In this setting the martingale transports correspond to measures $\pi$ on $\R^2$ with univariate marginals $\mu$ and $\nu$, and such that the kernel in the disintegration of $\pi$ satisfies the barycenter property
$$
\int y\pi_x(dy)=x,\quad\textrm{for }\mu\textrm{-almost every }x\in\R.
$$
Unlike in the classical setting, for arbitrary probability measures $\mu$ and $\nu$ the existence of a martingale transport is not guaranteed and requires an additional condition that $\mu$ is smaller than $\nu$ in convex order, which we will write as $\mu\leq_{cx}\nu$. That this condition is necessary and sufficient for there to exist a martingale with given marginals was proved by Strassen \cite{Strassen:65} (and extended to a level of continuous time processes by Kellerer \cite{Kell:72,Kell:73}).

Provided that $\mu\leq_{cx}\nu$, one then seeks to find optimal martingale couplings. For quadratic costs, however, the martingale transport problem is trivial. In particular, any martingale coupling is an optimiser. Solutions are known for several other specific but important costs. Hobson and Neuberger \cite{HobsonNeuberger:12} and Hobson and Klimmek \cite{HobsonKlimmek:15}, in the context of mathematical finance, provide the (non-explicit and explicit, respectively) constructions of the optimal martingale couplings $\pi^{HN}$ and $\pi^{HK}$ for the cost functions $c(x,y)=-\lvert x-y\lvert$ and $c(x,y)=\lvert x-y \lvert$, respectively. (Hobson and Klimmek~\cite{HobsonKlimmek:15} work under the dispersion assumption whereby $\mu \geq \nu$ on an interval $I$ and $\mu \leq \nu$ outside $I$.) Beiglb\"{o}ck and Juillet \cite{BeiglbockJuillet:16} introduced the so-called left-curtain coupling $\pi^{lc}$ (a martingale counterpart of the quantile coupling in the classical optimal transport) and proved its optimality for costs of the form $c(x,y)=h(y-x)$ for some differentiable function $h$ with strictly convex derivative.

All of the aforementioned martingale couplings have nice structural properties: if $\mu$ is atom-free then card(spt$\{\pi_x^{HN}\}$)$\leq2$, card(spt$\{\pi_x^{HK}\}$)$\leq3$ with card(spt$\{\pi_x^{HK}\}\setminus\{x\}$)$\leq2$ and card(spt$\{\pi_x^{lc}\}$)$\leq2$ for $\mu$-almost every $x\in\R$. In particular, in each case there exist lower and upper functions on which the couplings are concentrated. Under the dispersion assumption, Hobson and Klimmek \cite{HobsonKlimmek:15} constructed the upper and lower functions for $\pi^{HK}$, while for $\pi^{HN}$ only the existence is known. When the initial law $\mu$ is atomic, an explicit construction of the characteristic functions of the left-curtain coupling is provided in Beiglb\"{o}ck and Juillet \cite{BeiglbockJuillet:16}. Another construction of $\pi^{lc}$ (using ordinary differential equations) is given by Henry-Labord\`{e}re and Touzi \cite{HenryLabordereTouzi:16} for atomless initial measures $\mu$. For general initial and target laws Hobson and Norgilas \cite{HobsonNorgilas:18} constructed the upper and lower functions that characterise the generalised (or \textit{lifted}) left-curtain martingale coupling  using weak approximation of measures. Several other authors further investigate the properties and extensions of the left-curtain coupling, see Beiglb\"{o}ck et al. \cite{BeiglbockHenryLabordereTouzi:17,BeiglbockCox:17}, Juillet \cite{Juillet:16}, Nutz et al. \cite{NutzStebegg:18,NutzStebeggTan:17}, Campi et al. \cite{Campi:17}. In the case of a continuum of marginals which are increasing in convex order, Henry-Labord\`{e}re et al.~\cite{HenryLabordereTanTouzi:16}, Juillet~\cite{Juillet:18} and Br\"{u}ckerhoff at al.~\cite{BHJ:20} recently showed, amongst other things, how using the left-curtain coupling one can construct a martingale that fits given marginals at any given time and solves a continuous-time version of the martingale optimal transport problem.

To study the transport plans in the martingale setting, Beiglb\"{o}ck and Juillet \cite{BeiglbockJuillet:16} introduced the \textit{extended} convex order of two measures. Its significance lies in the fact that, for each pair of measures $\mu,\nu$ with $\mu$ less than $\nu$ in the extended convex order, there exists a martingale that transports $\mu$ into $\nu$ (without necessarily covering all of $\nu$). In particular, the set of measures $\eta$ with  $\mu\leq_{cx}\eta \leq \nu$, is non-empty. Each such $\eta$ corresponds to a terminal law of a martingale that embeds $\mu$ into $\nu$. Among these terminal laws there are at least two canonical choices, namely, the smallest and the largest element with respect to the convex order, which correspond to the most concentrated and the most disperse attainable terminal law of the transporting martingale, respectively. Beiglb\"{o}ck and Juillet \cite{BeiglbockJuillet:16} proved the existence and uniqueness of both of these extreme measures and baptised the minimal measure, which is the main object of interest in this paper, the \textit{shadow} (of $\mu$ in $\nu$). One of the main achievements of this paper is \cref{thm:shadow_potential}, which provides an explicit construction of the shadow measure. This allows us to give a simple proof of the existence and uniqueness which is considerably more direct than that given in \cite{BeiglbockJuillet:16}.

Arguably the most important structural property of the shadow measure is its associativity (see \cref{thm:shadow_assoc}).  In particular, it is the main ingredient in defining the \textit{left-curtain} martingale coupling, or more generally, any martingale coupling induced by the shadow measure. More specifically, the left-curtain martingale coupling is defined as the unique measure $\pi^{lc}$ on $\R^2$ such that, for each $x\in\R$, $\pi^{lc}\lvert_{(-\infty,x])\times\R}$ has the first marginal $\mu\lvert_{(-\infty,x]}$ and the second marginal $S^\nu(\mu\lvert_{(-\infty,x]})$, where $S^\nu(\mu\lvert_{(-\infty,x]})$ is the shadow of $\mu\lvert_{(-\infty,x]}$ in $\nu$. Thus $\pi^{lc}$ can be viewed as one extreme of the family of martingale transports constructed using the shadow measure, which corresponds to a \textit{horizontal} parametrisation $(\mu\lvert_{(-\infty,x]})_{x\in\R}$ of the initial measure $\mu$. Another such parametrisation, which corresponds to the (lifted) \textit{middle-curtain} ($\pi^{mc}$) coupling, is given by $(S^\mu(u\delta_{\overline\mu}))_{u\in[0,1])}$, where $\overline\mu=\int x\mu(dx)$. At the other end of the spectrum using the \textit{vertical} parametrisation $(u\mu)_{u\in[0,1]}$,  Beiglb\"{o}ck and Juillet \cite{BeiglbockJuillet:16s} obtained a rather different (lifted) canonical transport plan, namely, the \textit{sunset} ($\pi^{sun}$) coupling.

 In general, lifted couplings do not appear as the optimisers in optimal transport problems with martingale constraints (e.g., the sunset coupling $\pi^{sun}$ can be seen as the martingale analogue of the product coupling $\mu\otimes\nu$ in classical optimal transport). However, they are optimal in the general (or \textit{weak}) formulation of the transportation problem recently introduced by Gozlan et al. \cite{Gozlan:17} (see also Gozlan and Juillet \cite{Gozlan:18} and Backhoﬀ-Veraguas et al. \cite{Backhoff:19}). A further important property enjoyed by $\pi^{sun}$ is \textit{Lipschitz-Markovianity}, which is the main ingredient in all the proofs (to the best of our knowledge) of Kellerer's Theorem \cite{Kell:72} on the existence of Markov martingales with given marginals.

The current paper is structured as follows. In \cref{sec:prelims} we discuss the relevant notions of probability measures and (extended) convex order, and prove some crucial (for our main theorems) results regarding the convex hull of a function. In \cref{sec:max} we use potential-geometric arguments to explicitly construct the maximal measure with respect to convex order (the opposite of the shadow measure). \cref{sec:shadow} is dedicated to our main results. First, in \cref{thm:shadow_potential} we provide an explicit construction of the shadow measure in terms of its potential function. Then we use this result to give a simplified proof of the associativity of the shadow, see \cref{thm:shadow_assoc}.

\section{Preliminaries}\label{sec:prelims}
\subsection{Measures and Convex order}
\label{prelims:convex}
Let $\sM$ (respectively $\sP$) be the set of measures (respectively probability measures) on $\R$ with finite total mass and finite first moment, i.e., if $\eta\in\sM$, then $\eta(\R)<\infty$ and $\int_\R\lvert x\lvert\eta(dx)<\infty$. Given a measure $\eta\in\sM$ (not necessarily a probability measure), define $\overline{\eta} = \int_\R x \eta(dx)$ to be the first moment of $\eta$ (and then $\overline{\eta}/\eta(\R)$ is the barycentre of $\eta$). Let $\sI_\eta$ be the smallest interval containing the support of $\eta$, and let $\{ \ell_\eta, r_\eta \}$ be the endpoints of $\sI_\eta$. If $\eta$ has an atom at $\ell_\eta$ then $\ell_\eta$ is included in $\sI_\eta$, and otherwise it is excluded, and similarly for $r_\eta$. 

For $\alpha \geq 0$ and $\beta \in \R$ let $\sD(\alpha, \beta)$ denote the set of increasing, convex functions $f:\R \mapsto \R_+$ such that
\[ \lim_{ z \downarrow -\infty}  \{ f(z) \} =  0, \hspace{10mm} \lim_{z \uparrow \infty} \{ f(z) - (\alpha z- \beta) \}   =0. \]
Then, when $\alpha = 0$, $\sD(0,\beta)$ is empty unless $\beta = 0$ and then $\sD(0,0)$
contains one element, the zero function.

For $\eta\in\sM$, define the functions $P_\eta,C_\eta : \R \mapsto \R^+$ by
\begin{equation*}
P_\eta(k) := \int_{\R} (k-x)^+ \eta(dx),\quad k\in\R,
\end{equation*}
and
\begin{equation*}
C_\eta(k) := \int_{\R} (x-k)^+ \eta(dx),\quad k\in\R,
\end{equation*}
respectively. Then $P_\eta(k) \geq 0 \vee  (\eta(\R) k - \overline{\eta} )$ and $C_\eta(k) \geq 0 \vee (\overline{\eta} - \eta(\R)k)$. Also $C_\eta(k) - P_\eta(k) = (\overline{\eta}- \eta(\R)k)$.

The following properties of $P_\eta$ can be found in Chacon~\cite{Chacon:77}, and Chacon and Walsh~\cite{ChaconWalsh:76}: $P_\eta \in \sD(\eta(\R), \overline{\eta})$ and $\{k : P_{\eta}(k) > (\eta(\R)k - \overline{\eta})^+  \} =  \{k : C_{\eta}(k) > (\overline{\eta}- \eta(\R)k)^+\} =(\ell_\eta,r_\eta)$. Conversely (see, for example, Proposition 2.1 in Hirsch et al. \cite{peacock}),  if $h$ is a non-negative, non-decreasing and convex function with $h \in \sD(k_m,k_f)$ for some numbers $k_m \geq 0$ and $k_f\in\R$ (with $k_f = 0$ if $k_m=0$), then there exists a unique measure $\eta\in\sM$, with total mass $\eta(\R)=k_m$ and first moment $\overline{\eta}=k_f$, such that $h=P_{\eta}$. In particular, $\eta$ is uniquely identified by the second derivative of $h$ in the sense of distributions. Furthermore, $P_\eta$ and $C_\eta$ are related to the potential $U_\eta$, defined by
\begin{equation*}
U_\eta(k) : =  - \int_{\R} |k-x| \eta(dx),\quad k\in\R,
\end{equation*}
by $-U_\eta=C_\eta+P_\eta$. We will call $P_\eta$ (and $C_\eta$) a modified potential. Finally note that all three second derivatives $C^{\prime\prime}_{\eta},P^{\prime\prime}_{\eta}$ and $-U_\eta^{\prime\prime}/2$ identify the same underlying measure  $\eta$.

For $\eta,\chi\in\sM$, let
$$
\sC(\eta,\chi):=\{\tilde{P} \in \sD( \eta(\R), \overline{\eta}):P_\chi-\tilde{P}\textrm{ is convex and }P_{\eta}\leq \tilde P\}.
$$

For $\eta,\chi\in\sM$, we write $\eta\leq\chi$ if $\eta(A) \leq \chi(A)$ for all Borel measurable subsets $A$ of $\R$, or equivalently if
\begin{equation*}
\int fd\eta\leq\int fd\chi,\quad \textrm{for all non-negative }f:\R\mapsto\R_+.
\end{equation*}
Since $\eta$ and $\chi$ can be identified as second derivatives of $P_\chi$ and $P_\eta$ respectively, we have $\eta\leq\chi$ if and only if $P_\chi-P_\eta$ is convex, i.e., $P_\eta$ has a smaller curvature than $P_\chi$.

Two measures $\eta,\chi\in\sM$ are in convex order, and we write $\eta \leq_{cx} \chi$, if
\begin{equation}\label{eq:cx}
\int fd\eta\leq\int fd\chi,\quad\textrm{for all convex }f:\R\mapsto\R.
\end{equation}
Since we can apply \eqref{eq:cx} to all affine functions, including $f(x)=\pm 1$ and $f(x)=\pm x$, we obtain that
if $\eta \leq_{cx} \chi$ then $\eta$ and $\chi$ have the same total mass ($\eta(\R)= \chi(\R)$) and the same first moment ($\overline{\eta}= \overline{\chi}$).
Moreover, necessarily we must have $\ell_\chi \leq \ell_\eta \leq r_\eta \leq r_\chi$. From simple approximation arguments (see Hirsch et al. \cite{peacock}) we also have that if $\eta$ and $\chi$ have the same total mass and the same barycentre, then $\eta \leq_{cx} \chi$ if and only if $P_{\eta}(k) \leq P_{\chi}(k)$, $k\in\R$.

For our purposes in the sequel we need a generalisation of the convex order of two measures. We follow Beiglb\"{o}ck and Juillet \cite{BeiglbockJuillet:16} and say $\eta,\chi\in\sM$ are in an \textit{extended} convex order, and write $\eta\leq_{E}\chi$, if
\begin{equation*}
\int fd\eta\leq\int fd\chi,\quad\textrm{for all non-negative, convex }f:\R\mapsto\R_+.
\end{equation*}
If $\eta\leq_{cx}\chi$ then also $\eta\leq_E\chi$ (since non-negative convex functions are convex), while if $\eta\leq\chi$, we also have that $\eta\leq_E\chi$ (since non-negative convex functions are non-negative). Note that, if $\eta\leq_E\chi$, then $\eta(\R)\leq\chi(\R)$ (apply the non-negative convex function $\phi(x)=1$ in the definition of $\leq_E$). It is also easy to prove that, if $\eta(\R)=\chi(\R)$, then $\eta\leq_{E}\chi$ is equivalent to $\eta\leq_{cx}\chi$.

For $\eta,\chi\in\sP$, let ${\Pi}(\eta,\chi)$ be the set of probability measures on $\R^2$ with the first marginal $\eta$ and second marginal $\chi$. Let ${\Pi}_M(\eta,\chi)$ be the set of martingale couplings of $\eta$ and $\chi$.
Then
\begin{equation*}
{\Pi}_M(\eta,\chi) = \big\{ \pi \in {\Pi}(\eta,\chi) : \mbox{\eqref{eq:martingalepi} holds} \big\},
\end{equation*}
where \eqref{eq:martingalepi} is the martingale condition
\begin{equation}
\int_{x \in B} \int_{y \in \R} y \pi(dx,dy) = \int_{x \in B} \int_{y \in \R} x \pi(dx,dy) = \int_B x \eta(dx), \quad
\mbox{$\forall$ Borel $B \subseteq \R$}.
\label{eq:martingalepi}
\end{equation}
Equivalently, $\Pi_M(\eta,\chi)$ consists of all transport plans $\pi$ (i.e., elements of ${\Pi}(\eta,\chi)$) such that the disintegration in probability measures $(\pi_x)_{x\in\R}$ with respect to $\eta$ satisfies $\int_\R y\pi_x(dy)=x$ for $\eta$-almost every $x$. 

If we ignore the martingale requirement \eqref{eq:martingalepi}, it is easy to see that the set of probability measures with given marginals is non-empty, i.e., ${\Pi}(\eta,\chi)\neq\emptyset$ (consider the product measure $\eta\otimes\chi$). However, the fundamental question whether, for given $\eta$ and $\chi$, the set of martingale couplings $\Pi_M(\eta,\chi)$ is non-empty, is more delicate. For any $\pi\in\Pi_M(\eta,\chi)$ and convex $f:\R\mapsto\R$, by (conditional) Jensen's inequality we have that
\begin{equation*}
\int_\R f(x)\eta(dx)\leq \int_\R\int_\R f(y)\pi_x(dy)\eta(dx)=\int_\R f(y)\pi(\R,dy)=\int_\R f(y)\chi(dy),
\end{equation*}
so that $\eta\leq_{cx}\chi$. On the other hand, Strassen \cite{Strassen:65} showed that a converse is also true (i.e., $\eta\leq_{cx}\chi$ implies that ${\Pi}_M(\eta,\chi)\neq\emptyset$), so that ${\Pi}_M(\eta,\chi)$ is non-empty if and only if $\eta\leq_{cx}\chi$.

\subsection{Convex hull}\label{prelims:hull}

Our key results will be expressed in terms of the convex hull. For $f : \R \mapsto (-\infty,\infty)$ let $f^c$ be the largest convex function which lies below $f$. In our typical application $f$ will be non-negative and this property will be inherited by $f^c$. However, in general we may have $f^c$ equal to $-\infty$ on $\R$, and the results of this section are stated in a way which includes this case.  Note that if a function $g$ is equal to $-\infty$ (or $\infty$) everywhere, then we deem it to be both linear and convex, and set $g^c$ equal to $g$.

Fix $x,z\in\R$ with $x \leq z$. For $x<z$, define $L^f_{x,z}:[x,z]\mapsto\R$ by
\begin{equation}\label{eq:L1}
L^f_{x,z}(y)=f(x)\frac{z-y}{z-x}+f(z)\frac{y-x}{z-x},\quad y\in[x,z]
\end{equation}
and for $x=z$, define $L^f_{x,x}:\{x\}\mapsto\R$ by
$L^f_{x,x}(x)=f(x)$.
Then, see Rockafellar \cite[Corollary 17.1.5]{Rockafellar:72},
\begin{equation}\label{eq:chull}
f^c(y)=\inf_{x\leq y\leq z}L^f_{x,z}(y),\quad y\in\R.
\end{equation}

Moreover, it is not hard to see (at least pictorially, by drawing the graphs of $f$ and $f^c$) that $f^c$ replaces the non-convex segments of $f$ by straight lines. (Proofs of lemmas in this section are given in \cref{sec:proofs}.)

\begin{lem}
	\label{lem:linear}
	Let $f:\R\mapsto\R$ be a lower semi-continuous function. Suppose $f>f^c$ on $(a,b)\subseteq\R$. Then $f^c$ is linear on $(a,b)$.
\end{lem}

The following lemmas are the main ingredients in the proofs of \cref{thm:shadow_potential} and \cref{thm:shadow_assoc}.
\begin{lem}\label{lem:diff}
	Let $f,g:\R\mapsto\R$ be two convex functions. Define $\psi:\R\mapsto(-\infty,\infty]$ by
	$
	\psi =g-(g-f)^c
	$.
	Then $\psi$ is convex.
\end{lem}

\begin{lem}\label{lem:equal_hulls}
	Let $f:\R\mapsto\R$ be any measurable function and let $g:\R\mapsto\R$ be a convex function. Then
	$$
	(f-g)^c=(f^c-g)^c.
	$$
\end{lem}

\begin{lem}\label{CSpotential} Assume that
	$f\in \sD(\alpha, \beta)$ and $g\in \sD(a,b)$ for some $\alpha, a \geq 0$, $\beta, b \in \R$. Let $h:\R\mapsto\R$ be defined by $h(k):=(a-\alpha)k-(b-\beta)$.
	
	Suppose that $g\geq f$. Then $\alpha \leq a$. If $\alpha = a$ then $\beta \geq b$.
	
	
	Suppose that $g \geq f$ and $g-f\geq h$. Then $ (g-f)^c\in \sD(a-\alpha, {b-\beta}).$
\end{lem}

Note that in the above lemma if $\alpha=a$ and $\beta>b$ then there are no pairs of functions $(f,g)$ with $f\in \sD(a, \beta)$, $g\in \sD(a,b)$  and $g(z)-f(z) \geq h(z) = \beta - b$. We cannot have both $\lim_{z \downarrow -\infty} (g(z)-f(z))=0$ and $g(z)-f(z) \geq h(z) = \beta - b>0$. In this case the final statement of the lemma is vacuous.

One main use of Lemma~\ref{CSpotential} is via the following corollary which follows immediately when Lemma~\ref{CSpotential} is combined with Corollary~\ref{corE}:
\begin{cor} \label{cor:CSpotential}
	Suppose $\mu \leq_E \nu$. Then $(P_{\nu} - P_{\mu})^c \in \sD ( \nu(\R)-\mu(\R), \overline{\nu}- \overline{\mu}) $.
\end{cor}

Note that if $\mu \leq_E \nu$ and $\mu(\R) = \nu(\R)$ then $\mu \leq_{cx} \nu$ and $\overline{\mu} = \overline{\nu}$ so that $(P_{\nu} - P_{\mu})^c$ is the zero function which is the unique element in $\sD(0,0)$.

\section{The maximal element}\label{sec:max}
Let $\mu,\nu\in\sM$ with $\mu\leq_E\nu$ and define $\sM^\nu_\mu=\{\eta\in\sM:\mu\leq_{cx}\eta\leq\nu\}$. Then $\sM^\nu_\mu$ is a set of terminal laws of a martingale that embeds $\mu$ into $\nu$. Note that $\eta\in\sM^\nu_\mu$ if and only if $P_\eta\in\sC(\mu,\nu)$.

In this section we show that $\sM^\nu_\mu\neq\emptyset$. In fact we find the largest measure (w.r.t. convex order) in $\sM^\nu_\mu$ (see \cref{thm:maximal}).
\begin{defn}[Counter-shadow measure]\label{defn:maximal}
	Let $\mu,\nu \in\sM$ and assume $\mu \leq_E \nu$. The counter-shadow of $\mu$ in $\nu$, denoted by $T^{\nu}(\mu)$, has the following properties:
	\begin{enumerate}
		\item\label{item1} $\mu\leq_{cx}T^\nu(\mu)$,
		\item\label{item2} $T^\nu(\mu)\leq\nu$,
		\item\label{item3} If $\eta$ is another measure satisfying $\mu \leq_{cx} \eta \leq \nu$ then $\eta\leq_{cx}T^\nu(\mu)$.
	\end{enumerate}
\end{defn}

\begin{rem}\label{rem:Ecx0}
	If $\mu \leq_{cx} \nu$ then necessarily $T^\nu(\mu)=\nu$, since $\{ \eta: \mu \leq_{cx} \eta \leq \nu \}$ is the singleton $\{\nu\}$.
\end{rem}

\begin{lem}[Beiglb\"{o}ck and Juillet \cite{BeiglbockJuillet:16}, Lemma 4.5]\label{lem:max}
	For $\mu, \nu \in \sM$ with $\mu \leq_E \nu$, $T^\nu(\mu)$ exists and is unique.
\end{lem}

Beiglb\"{o}ck and Juillet \cite{BeiglbockJuillet:16} not only prove the existence and uniqueness of $T^{\nu}(\mu)$, but also show how to construct $T^{\nu}(\mu)$. Let $G_\nu:[0,\nu(\R)]\mapsto\R$ be a quantile function of $\nu$. Then each $\zeta\in[0,\mu(\R)]$ defines a measure
$$
\theta^\zeta=\nu\lvert_{(-\infty,G(\zeta))}+\nu\lvert_{(G(\zeta+\nu(\R)-\mu(\R)),\infty)}+\alpha^\zeta\delta_{G(\zeta)}
+\beta^\zeta\delta_{G(\zeta+\nu(\R)-\mu(\R))},
$$
where $0\leq\alpha^\zeta = \zeta - \nu((-\infty,G(\zeta))) \leq\nu(\{G(\zeta)\})$ and $0\leq\beta^\zeta = \mu(\R) - \zeta  - \nu((G(\zeta+\nu(\R)-\mu(\R)),\infty)) \leq\nu(\{G(\zeta+\nu(\R)-\mu(\R))\})$. By construction, $\theta^\zeta(\R)=\mu(\R)$.  Furthermore, $\overline{\theta^{0}}\geq\overline\mu\geq \overline{\theta^{\mu(\R)}}$ and $\overline{\theta^{\zeta}}$ is continuous and decreasing in $\zeta$, and therefore there exists $\zeta_*$ such that $\overline{\theta^{\zeta_*}}=\overline{\mu}$. Then, since $\theta^{\zeta_*}$ is concentrated in the tails of $\nu$, it can be shown that $\theta^{\zeta_*}$ satisfies all three properties in \cref{defn:maximal}.

Our first result provides an alternative explicit construction of $T^\nu(\mu)$ via potential functions. Let $\tilde{P}^{\nu}_\mu:\R\mapsto\R$ be given by
$$
\tilde{P}^{\nu}_\mu(k)=\min\{P_\nu(k),C_\nu(k) +  (\mu(\R)k-\overline{\mu})\},\quad k\in\R.
$$

\begin{thm}\label{thm:maximal} Suppose $\mu,\nu \in \sM$ with $\mu \leq_E \nu$. Then
	$$
	P_{T^\nu(\mu)}=(\tilde{P}^{\nu}_\mu)^c.
	$$
	In particular, $T^\nu(\mu)$ corresponds to the second derivative of $(\tilde{P}^{\nu}_\mu)^c$ in the sense of distributions.
\end{thm}

\begin{proof} For $\nu(\R) = \mu(\R)$ we must have $\bar{\nu} = \bar{\mu}$ and $T^\nu(\mu) = \nu$. The result can be verified directly in this case and we exclude it from this point onwards.
	
	Note first that since $\mu \leq_E \nu$, $P_\nu\geq P_\mu$ and $C_\nu(k) +  (\mu(\R)k-\overline{\mu}) \geq C_\mu(k)+ (\mu(\R)k-\overline{\mu}) = P_\mu(k)$. Then $\tilde{P}^{\nu}_\mu \geq P_\mu\geq 0$.
	
	Let $P_{\chi}=(\tilde{P}^{\nu}_\mu)^c$ be our candidate function. Since $P_\chi$ is convex, its second derivative (in the sense of distributions) corresponds to a measure, which we denote by $\chi$.
	
	We first show that $P_{T^\nu(\mu)}\leq P_\chi$. Since $P_{T^\nu(\mu)}$ is convex and $P_\chi$ is the largest convex function below $\tilde{P}^{\nu}_\mu$, it is enough to show that $P_{T^\nu(\mu)}\leq \tilde{P}^{\nu}_\mu$. Since $T^\nu(\mu)\leq\nu$, $P_{T^\nu(\mu)}\leq P_{\nu}$. On the other hand, since $\mu\leq_{cx}T^\nu(\mu)$, $\mu(\R)=T^\nu(\mu)(\R)$ and $\overline{\mu}=\overline{T^\nu(\mu)}$, and therefore, for all $k\in\R$, $P_{T^\nu(\mu)}(k)=C_{T^\nu(\mu)}(k)+ (\mu(\R) k-\overline\mu)\leq C_{\nu}(k)+(\mu(\R)k-\overline\mu)$, where we again used that $T^\nu(\mu)\leq\nu$. Combining both cases we conclude that $P_{T^\nu(\mu)}\leq \tilde{P}^{\nu}_\mu$ and thus $P_{T^\nu(\mu)}\leq P_\chi$.
	
	To finish the proof we will show that $\chi\in\sM^\nu_\mu$, or equivalently, that $P_\chi\in\sC(\mu,\nu)$. Then by the maximality of $T^\nu(\mu)$ we have that $\chi\leq_{cx}T^\nu(\mu)$, which implies that $P_\chi\leq P_{T^\nu(\mu)}$.
	
	First, we already saw that $P_\mu\leq\tilde{P}^\nu_\mu$. Then, since $P_\mu$ is convex, $P_\mu\leq  (\tilde{P}^\nu_\mu)^c = P_\chi\leq\tilde{P}^\nu_\mu$. Further, since $P_\mu$ is an element of $\sD(\mu(\R), \overline{\mu})$ and $\tilde{P}^\nu_\mu$ has the same limiting behaviour as an element of $\sD(\mu(\R), \overline{\mu})$, it follows that $P_{\chi} \in \sD(\mu(\R), \overline{\mu})$, and therefore $\mu\leq_{cx}\chi$.
	
	It is left to show that $\chi\leq\nu$, or equivalently, that $P_\nu-P_\chi$ is convex. But $\tilde{P}^\nu_\mu(k) = P_\nu(k) - ((\nu(\R)-\mu(\R))k - (\bar{\nu} - \bar{\mu}))^+$. But $p$ given by $p(k)=((\nu(\R)-\mu(\R))k - (\bar{\nu} - \bar{\mu}))^+$ is a convex function and therefore by Lemma~\ref{lem:diff}, with $g=P_\nu$ and $f=p$, $P_\nu - (P_\nu - p)^c = P_\nu - P_\chi$ is convex.
\end{proof}

\begin{figure}[H]\centering
	\begin{tikzpicture}[scale=0.8]
	
	\begin{axis}[%
	width=5.028in,
	height=3.5in,
	at={(1.011in,0.642in)},
	scale only axis,
	xmin=-3,
	xmax=3,
	ymin=-1,
	ymax=4,
	axis line style={draw=none},
	ticks=none
	]
	\addplot [color=black, line width=0.5pt, forget plot]
	table[row sep=crcr]{%
		-3	0\\
		-2.87486853728013	0\\
		-2.72727272727273	0\\
		-2.58450415319679	0\\
		-2.45454545454545	0\\
		-2.32075022301553	0\\
		-2.18181818181818	0\\
		-2.03135535207571	0\\
		-1.90909090909091	0\\
		-1.75711081622359	0\\
		-1.63636363636364	0\\
		-1.50480982849432	0\\
		-1.36363636363636	0\\
		-1.23254829617911	0\\
		-1.09090909090909	0\\
		-0.938793844083881	0\\
		-0.818181818181818	0\\
		-0.671931194594317	0\\
		-0.545454545454545	0\\
		-0.401923560569463	0\\
		-0.272727272727273	0\\
		-0.137628516354142	0\\
		0	0\\
		0.136399391257762	0\\
		0.272727272727273	0\\
		0.414511205435068	0\\
		0.545454545454545	0\\
		0.692747678649345	0\\
		0.818181818181818	0\\
		0.952205434434925	0\\
		1.09090909090909	0\\
		1.22816802555557	0\\
		1.36363636363636	0\\
		1.50010537849835	0\\
		1.63636363636364	0\\
		1.7826738865474	0\\
		1.90909090909091	0\\
		2.03902277357478	0\\
		2.18181818181818	0\\
		2.3249118553759	0\\
		2.45454545454545	0\\
		2.55747419264388	0\\
		2.72727272727273	0\\
		2.86110314258923	0\\
		3	0\\
	};
	\addplot [color=black, line width=0.5pt, forget plot]
	table[row sep=crcr]{%
		-3	3\\
		-2.87486853728013	2.87486853728013\\
		-2.72727272727273	2.72727272727273\\
		-2.58450415319679	2.58450415319679\\
		-2.45454545454545	2.45454545454545\\
		-2.32075022301553	2.32075022301553\\
		-2.18181818181818	2.18181818181818\\
		-2.03135535207571	2.03135535207571\\
		-1.90909090909091	1.90909090909091\\
		-1.75711081622359	1.75711081622359\\
		-1.63636363636364	1.63636363636364\\
		-1.50480982849432	1.50480982849432\\
		-1.36363636363636	1.36363636363636\\
		-1.23254829617911	1.23254829617911\\
		-1.09090909090909	1.09090909090909\\
		-0.938793844083881	0.938793844083881\\
		-0.818181818181818	0.818181818181818\\
		-0.671931194594317	0.671931194594317\\
		-0.545454545454545	0.545454545454545\\
		-0.401923560569463	0.401923560569463\\
		-0.272727272727273	0.272727272727273\\
		-0.137628516354142	0.137628516354142\\
		-0.068814258177071	0.068814258177071\\
		-0.0344071290885355	0.0344071290885355\\
		-0.0172035645442677	0.0172035645442677\\
		-0.00860178227213387	0.00860178227213387\\
		-0.00430089113606694	0.00430089113606694\\
		-0.00215044556803347	0.00215044556803347\\
		-0.00107522278401673	0.00107522278401673\\
		-0.000537611392008367	0.000537611392008367\\
		-0.000268805696004184	0.000268805696004184\\
		0	0\\
		0.000266405061050316	0.000266405061050316\\
		0.000532810122100632	0.000532810122100632\\
		0.00106562024420126	0.00106562024420126\\
		0.00213124048840253	0.00213124048840253\\
		0.00426248097680505	0.00426248097680505\\
		0.00852496195361011	0.00852496195361011\\
		0.0170499239072202	0.0170499239072202\\
		0.0340998478144404	0.0340998478144404\\
		0.0681996956288808	0.0681996956288808\\
		0.136399391257762	0.136399391257762\\
		0.272727272727273	0.272727272727273\\
		0.414511205435068	0.414511205435068\\
		0.545454545454545	0.545454545454545\\
		0.692747678649345	0.692747678649345\\
		0.818181818181818	0.818181818181818\\
		0.952205434434925	0.952205434434925\\
		1.09090909090909	1.09090909090909\\
		1.22816802555557	1.22816802555557\\
		1.36363636363636	1.36363636363636\\
		1.50010537849835	1.50010537849835\\
		1.63636363636364	1.63636363636364\\
		1.7826738865474	1.7826738865474\\
		1.90909090909091	1.90909090909091\\
		2.03902277357478	2.03902277357478\\
		2.18181818181818	2.18181818181818\\
		2.3249118553759	2.3249118553759\\
		2.45454545454545	2.45454545454545\\
		2.55747419264388	2.55747419264388\\
		2.72727272727273	2.72727272727273\\
		2.86110314258923	2.86110314258923\\
		3	3\\
	};
	\addplot [color=black, line width=0.5pt, forget plot]
	table[row sep=crcr]{%
		-3	-0\\
		-2.87486853728013	-0\\
		-2.72727272727273	-0\\
		-2.58450415319679	-0\\
		-2.45454545454545	-0\\
		-2.32075022301553	-0\\
		-2.18181818181818	-0\\
		-2.03135535207571	-0\\
		-1.90909090909091	-0\\
		-1.75711081622359	-0\\
		-1.63636363636364	-0\\
		-1.50480982849432	-0\\
		-1.36363636363636	-0\\
		-1.23254829617911	-0\\
		-1.09090909090909	-0\\
		-0.938793844083881	-0\\
		-0.818181818181818	-0\\
		-0.671931194594317	-0\\
		-0.545454545454545	-0\\
		-0.401923560569463	-0\\
		-0.272727272727273	-0\\
		-0.137628516354142	-0\\
		0	-0\\
		0.136399391257762	-0\\
		0.272727272727273	-0\\
		0.414511205435068	-0\\
		0.447247040439937	-0\\
		0.479982875444807	-0\\
		0.488166834196024	-0\\
		0.492258813571633	-0\\
		0.496350792947241	-0\\
		0.498396782635046	-0\\
		0.499419777478948	-0\\
		0.499675526189923	-0\\
		0.499931274900899	-0\\
		0.500187023611874	3.27291320780187e-05\\
		0.50044277232285	7.74851564987411e-05\\
		0.500954269744801	0.000166997205340186\\
		0.501465767166752	0.00025650925418163\\
		0.502488762010654	0.0004355333518645\\
		0.504534751698459	0.00079358154723026\\
		0.512718710449676	0.0022257743286933\\
		0.529086627952111	0.00509015989161937\\
		0.545454545454545	0.00795454545454545\\
		0.619101112051945	0.0208426946090904\\
		0.692747678649345	0.0337308437636354\\
		0.818181818181818	0.0556818181818182\\
		0.952205434434925	0.0791359510261119\\
		1.09090909090909	0.103409090909091\\
		1.22816802555557	0.127429404472224\\
		1.36363636363636	0.151136363636364\\
		1.50010537849835	0.175018441237212\\
		1.63636363636364	0.198863636363636\\
		1.7826738865474	0.224467930145795\\
		1.90909090909091	0.246590909090909\\
		2.03902277357478	0.269328985375587\\
		2.18181818181818	0.294318181818182\\
		2.3249118553759	0.319359574690783\\
		2.45454545454545	0.342045454545454\\
		2.55747419264388	0.360057983712678\\
		2.72727272727273	0.389772727272727\\
		2.86110314258923	0.413193049953116\\
		3	0.4375\\
	};
	\addplot [color=black!60!green, line width=1.0pt, dotted, forget plot]
	table[row sep=crcr]{%
		-3	0\\
		-2.87486853728013	0\\
		-2.72727272727273	0\\
		-2.58450415319679	0\\
		-2.45454545454545	0\\
		-2.32075022301553	0\\
		-2.18181818181818	0\\
		-2.03135535207571	0\\
		-1.90909090909091	0.00103305785123967\\
		-1.75711081622359	0.00737439444944652\\
		-1.63636363636364	0.0165289256198347\\
		-1.50480982849432	0.0306516632444775\\
		-1.36363636363636	0.0506198347107438\\
		-1.23254829617911	0.0736227647121979\\
		-1.09090909090909	0.103305785123967\\
		-0.938793844083881	0.140769813169283\\
		-0.818181818181818	0.174586776859504\\
		-0.671931194594317	0.22047084398646\\
		-0.545454545454545	0.264462809917355\\
		-0.401923560569463	0.319231038282873\\
		-0.272727272727273	0.372933884297521\\
		-0.137628516354142	0.433553442887159\\
		0	0.5\\
		0.136399391257762	0.570525294870817\\
		0.272727272727273	0.645661157024794\\
		0.414511205435068	0.728733045146438\\
		0.545454545454545	0.809917355371901\\
		0.692747678649345	0.906361257608929\\
		0.818181818181818	0.992768595041322\\
		0.952205434434925	1.08943961588839\\
		1.09090909090909	1.19421487603306\\
		1.22816802555557	1.30263360015242\\
		1.36363636363636	1.41425619834711\\
		1.50010537849835	1.53134220757414\\
		1.63636363636364	1.65289256198347\\
		1.7826738865474	1.78857771649595\\
		1.90909090909091	1.91012396694215\\
		2.03902277357478	2.03902277357478\\
		2.18181818181818	2.18181818181818\\
		2.3249118553759	2.3249118553759\\
		2.45454545454545	2.45454545454545\\
		2.55747419264388	2.55747419264388\\
		2.72727272727273	2.72727272727273\\
		2.86110314258923	2.86110314258923\\
		3	3\\
	};
	\addplot [color=black!60!green, line width=1.0pt,  dotted, forget plot]
	table[row sep=crcr]{%
		-3	3\\
		-2.87486853728013	2.87486853728013\\
		-2.72727272727273	2.72727272727273\\
		-2.58450415319679	2.58450415319679\\
		-2.45454545454545	2.45454545454545\\
		-2.32075022301553	2.32075022301553\\
		-2.18181818181818	2.18181818181818\\
		-2.03135535207571	2.03135535207571\\
		-1.90909090909091	1.91012396694215\\
		-1.75711081622359	1.76448521067303\\
		-1.63636363636364	1.65289256198347\\
		-1.50480982849432	1.5354614917388\\
		-1.36363636363636	1.41425619834711\\
		-1.23254829617911	1.30617106089131\\
		-1.09090909090909	1.19421487603306\\
		-0.938793844083881	1.07956365725316\\
		-0.818181818181818	0.992768595041322\\
		-0.671931194594317	0.892402038580777\\
		-0.545454545454545	0.809917355371901\\
		-0.401923560569463	0.721154598852336\\
		-0.272727272727273	0.645661157024793\\
		-0.137628516354142	0.571181959241301\\
		0	0.5\\
		0.136399391257762	0.434125903613055\\
		0.272727272727273	0.372933884297521\\
		0.414511205435068	0.31422183971137\\
		0.545454545454545	0.264462809917355\\
		0.692747678649345	0.213613578959585\\
		0.818181818181818	0.174586776859504\\
		0.952205434434925	0.137234181453463\\
		1.09090909090909	0.103305785123967\\
		1.22816802555557	0.0744655745968492\\
		1.36363636363636	0.0506198347107438\\
		1.50010537849835	0.0312368290757843\\
		1.63636363636364	0.0165289256198347\\
		1.7826738865474	0.00590382994855165\\
		1.90909090909091	0.00103305785123964\\
		2.03902277357478	0\\
		2.18181818181818	0\\
		2.3249118553759	0\\
		2.45454545454545	0\\
		2.55747419264388	0\\
		2.72727272727273	0\\
		2.86110314258923	0\\
		3	0\\
	};
	\addplot [color=orange, loosely dash dot, line width=1.0pt, forget plot]
	table[row sep=crcr]{%
		-3	2.3875\\
		-2.87486853728013	2.28426654325611\\
		-2.72727272727273	2.1625\\
		-2.58450415319679	2.04471592638736\\
		-2.45454545454545	1.9375\\
		-2.32075022301553	1.82711893398781\\
		-2.18181818181818	1.7125\\
		-2.03135535207571	1.58836816546246\\
		-1.90909090909091	1.48853305785124\\
		-1.75711081622359	1.36949081783391\\
		-1.63636363636364	1.27902892561983\\
		-1.50480982849432	1.1846197717523\\
		-1.36363636363636	1.08811983471074\\
		-1.23254829617911	1.00297510905997\\
		-1.09090909090909	0.915805785123967\\
		-0.938793844083881	0.827774734538485\\
		-0.818181818181818	0.762086776859504\\
		-0.671931194594317	0.687314079526771\\
		-0.545454545454545	0.626962809917355\\
		-0.401923560569463	0.56331797575268\\
		-0.272727272727273	0.510433884297521\\
		-0.137628516354142	0.459596968879326\\
		0	0.4125\\
		0.136399391257762	0.370495797083163\\
		0.272727272727273	0.333161157024793\\
		0.414511205435068	0.299261300662507\\
		0.545454545454545	0.272417355371901\\
		0.692747678649345	0.24734442272322\\
		0.818181818181818	0.230268595041322\\
		0.952205434434925	0.216370132479575\\
		1.09090909090909	0.206714876033058\\
		1.22816802555557	0.201894979069073\\
		1.36363636363636	0.201756198347107\\
		1.50010537849835	0.206255270312996\\
		1.63636363636364	0.215392561983471\\
		1.7826738865474	0.230371760094347\\
		1.90909090909091	0.247623966942149\\
		2.03902277357478	0.269328985375587\\
		2.18181818181818	0.294318181818182\\
		2.3249118553759	0.319359574690783\\
		2.45454545454545	0.342045454545454\\
		2.55747419264388	0.360057983712678\\
		2.72727272727273	0.389772727272727\\
		2.86110314258923	0.413193049953116\\
		3	0.4375\\
	};
	\addplot [color=blue, line width=1.0pt, densely dashed, forget plot]
	table[row sep=crcr]{%
		-3	0\\
		-2.87486853728013	0\\
		-2.72727272727273	0\\
		-2.58450415319679	0\\
		-2.45454545454545	0\\
		-2.32075022301553	0\\
		-2.18181818181818	0\\
		-2.03135535207571	0\\
		-1.90909090909091	0.00103305785123967\\
		-1.75711081622359	0.00737439444944652\\
		-1.63636363636364	0.0165289256198347\\
		-1.50480982849432	0.0306516632444775\\
		-1.36363636363636	0.0506198347107438\\
		-1.23254829617911	0.0736227647121979\\
		-1.09090909090909	0.103305785123967\\
		-0.938793844083881	0.140769813169283\\
		-0.818181818181818	0.174586776859504\\
		-0.671931194594317	0.22047084398646\\
		-0.545454545454545	0.264462809917355\\
		-0.401923560569463	0.319231038282873\\
		-0.272727272727273	0.372933884297521\\
		-0.205177894540707	0.402673298780666\\
		-0.171403205447425	0.417970779630994\\
		-0.137628516354142	0.433553442887159\\
		-0.129026734082008	0.43756762022248\\
		-0.120424951809874	0.441600295222364\\
		-0.116124060673807	0.443623569346518\\
		-0.11182316953774	0.445651467886813\\
		-0.109672723969707	0.446667151313014\\
		-0.10859750118569	0.447175426565127\\
		-0.107522278401673	0.447683990843249\\
		-0.106984667009665	0.447938381367064\\
		-0.106715861313661	0.448065603725159\\
		-0.106447055617657	0.44819284414738\\
		-0.106178249921653	0.448320102633727\\
		-0.105909444225648	0.448322670670359\\
		-0.105640638529644	0.448228200585753\\
		-0.10537183283364	0.448133748565273\\
		-0.104834221441632	0.447944898716689\\
		-0.104296610049623	0.447756121124608\\
		-0.103221387265606	0.447378782709952\\
		-0.101070941697573	0.446624972958665\\
		-0.0989204961295395	0.445872319311415\\
		-0.0946196049934726	0.444370480329019\\
		-0.0860178227213387	0.44138067561265\\
		-0.068814258177071	0.435456559173606\\
		-0.0344071290885355	0.423830298270288\\
		0	0.4125\\
		0.0681996956288808	0.390916498731098\\
		0.136399391257762	0.370495797083163\\
		0.272727272727273	0.333161157024793\\
		0.414511205435068	0.299261300662507\\
		0.545454545454545	0.272417355371901\\
		0.692747678649345	0.24734442272322\\
		0.818181818181818	0.230268595041322\\
		0.952205434434925	0.216370132479575\\
		1.09090909090909	0.206714876033058\\
		1.22816802555557	0.201894979069073\\
		1.36363636363636	0.201756198347107\\
		1.50010537849835	0.206255270312996\\
		1.63636363636364	0.215392561983471\\
		1.7826738865474	0.230371760094347\\
		1.90909090909091	0.247623966942149\\
		2.03902277357478	0.269328985375587\\
		2.18181818181818	0.294318181818182\\
		2.3249118553759	0.319359574690783\\
		2.45454545454545	0.342045454545454\\
		2.55747419264388	0.360057983712678\\
		2.72727272727273	0.389772727272727\\
		2.86110314258923	0.413193049953116\\
		3	0.4375\\
	};
	\addplot [color=blue, line width=1.0pt, densely dashed,  forget plot]
	table[row sep=crcr]{%
		0.5	0.28125\\
		0.552138109466614	0.271162175914021\\
		0.613636363636364	0.260136880165289\\
		0.673123269501336	0.250371804405087\\
		0.727272727272727	0.242252066115703\\
		0.783020740410196	0.234658444355753\\
		0.840909090909091	0.22759555785124\\
		0.90360193663512	0.220891428079928\\
		0.954545454545455	0.216167355371901\\
		1.01787049324017	0.211199632323068\\
		1.06818181818182	0.207967458677686\\
		1.12299590479403	0.20516630621496\\
		1.18181818181818	0.202995867768595\\
		1.23643820992537	0.201755012644686\\
		1.29545454545455	0.201252582644628\\
		1.35883589829838	0.201682707866072\\
		1.40909090909091	0.202737603305785\\
		1.47002866891903	0.204863718531797\\
		1.52272727272727	0.207450929752066\\
		1.58253184976272	0.211228030766293\\
		1.63636363636364	0.215392561983471\\
		1.69265478485244	0.220522222508439\\
		1.75	0.2265625\\
		1.80683307969073	0.233359971333599\\
		1.86363636363636	0.240960743801653\\
		1.92271300226461	0.249721435398676\\
		1.97727272727273	0.25858729338843\\
		2.03864486610389	0.269262851568181\\
		2.09090909090909	0.278409090909091\\
		2.14675226434789	0.28818164626088\\
		2.20454545454545	0.298295454545455\\
		2.26173667731482	0.308303918530093\\
		2.31818181818182	0.318181818181818\\
		2.37504390770765	0.328132683848838\\
		2.43181818181818	0.338068181818182\\
		2.49278078606142	0.348736637560748\\
		2.54545454545455	0.357954545454545\\
		2.59959282232283	0.367428743906494\\
		2.65909090909091	0.377840909090909\\
		2.71871327307329	0.388274822787826\\
		2.77272727272727	0.397727272727273\\
		2.81561424693495	0.405232493213616\\
		2.88636363636364	0.417613636363636\\
		2.94212630941218	0.427372104147131\\
		3	0.4375\\
	};
	\addplot [color=red, line width=1.0pt, forget plot]
	table[row sep=crcr]{%
		-3	0\\
		-2.87486853728013	0\\
		-2.72727272727273	0\\
		-2.58450415319679	0\\
		-2.45454545454545	0\\
		-2.32075022301553	0\\
		-2.18181818181818	0\\
		-2.03135535207571	0\\
		-1.90909090909091	0.00103305785123967\\
		-1.75711081622359	0.00737439444944652\\
		-1.63636363636364	0.014738005050505\\
		-1.50480982849432	0.0227607940910199\\
		-1.36363636363636	0.0313702364554638\\
		-1.23254829617911	0.0393646223875158\\
		-1.09090909090909	0.0480024678604224\\
		-0.938793844083881	0.0572791931402932\\
		-0.818181818181818	0.0646346992653811\\
		-0.671931194594317	0.0735537713856795\\
		-0.545454545454545	0.0812669306703398\\
		-0.401923560569463	0.0900201460364377\\
		-0.272727272727273	0.0978991620752985\\
		-0.137628516354142	0.106138139263205\\
		0	0.114531393480257\\
		0.136399391257762	0.122849689689537\\
		0.272727272727273	0.131163624885216\\
		0.414511205435068	0.139810296538987\\
		0.545454545454545	0.147795856290175\\
		0.692747678649345	0.156778505700918\\
		0.818181818181818	0.164428087695133\\
		0.952205434434925	0.17260149762572\\
		1.09090909090909	0.181060319100092\\
		1.22816802555557	0.189431034432699\\
		1.36363636363636	0.197692550505051\\
		1.50010537849835	0.206015092699286\\
		1.63636363636364	0.215392561983471\\
		1.7826738865474	0.230371760094347\\
		1.90909090909091	0.247623966942149\\
		2.03902277357478	0.269328985375587\\
		2.18181818181818	0.294318181818182\\
		2.3249118553759	0.319359574690783\\
		2.45454545454545	0.342045454545454\\
		2.55747419264388	0.360057983712678\\
		2.72727272727273	0.389772727272727\\
		2.86110314258923	0.413193049953116\\
		3	0.4375\\
	};
	
	\draw [gray,->] (-0.5,1.5) to[out=230, in=350] (-1.8,1.5);
	
	\draw [gray, thin, densely dashed] (0,-0.2) -- (0,0.5);
	\draw [gray, thin, densely dashed] (0.5,0.25) -- (0.5,-0.2);
	\draw [gray, thin, densely dashed] (-1.7561,-0.2) -- (-1.7561,0);
	\draw [gray, thin, densely dashed] ( 1.5439,-0.2) -- ( 1.5439,0.2);
	\node (Cnu1)[scale=1.1,black!60!green] at (-2,2.7) {$k\mapsto C_{\nu}(k)$};
	\node (Cnu)[scale=1.1,orange] at (0,1.7) {$k\mapsto C_{\nu}(k)+(\mu(\R)k-\overline\mu)$};
	\node (Pnu)[scale=1.1,black!60!green] at ( 2,2.7) {$k\mapsto P_{\nu}(k)$};
	\node (Ptilde)[scale=1.1,blue] at (-1.1,0.5) {$k\mapsto \tilde{P}^{\nu}_\mu(k)$};
	\node (Phull)[scale=1.1,red] at ( 2,0.6) {$k\mapsto P_{T^\nu(\mu)}(k)$};
	\node (bmu)[scale=1.1] at ( 0.5,-0.4) {$ \frac{\overline\mu}{\mu(\R)}$};
	\node (bnu)[scale=1.1] at ( 0,-0.4) {$ \frac{\overline\nu}{\nu(\R)}$};
	\node (bmu1)[scale=1.1] at ( -1.7561,-0.3) {$k_*$};
	\node (bnu1)[scale=1.1] at ( 1.5439,-0.3) {$k^*$};
	\end{axis}
	\end{tikzpicture}
	\caption{Construction of $P_{T^\nu(\mu)}$. Dotted curves correspond to the graphs of $C_\nu$ and $P_\mu$, while the dash-dotted curve corresponds to $k\mapsto C_{\nu}(k)+(\mu(\R)k-\overline\mu)$. The dashed curve represents $\tilde{P}^\nu_\mu$. Note that $\tilde{P}^\nu_\mu$ is continuous. The solid curve below  $\tilde{P}^\nu_\mu$ corresponds to $P_{T^\nu(\mu)}$, the (modified) potential of the counter-shadow. Note that, on $(k_*,k^*)$, $P_{T^\nu(\mu)}$ is linear and strictly below $\tilde{P}^\nu_\mu$, while on $\R\setminus(k_*,k^*)$ it is equal to $\tilde{P}^\nu_\mu$.}
	\label{plot:LU}
\end{figure}
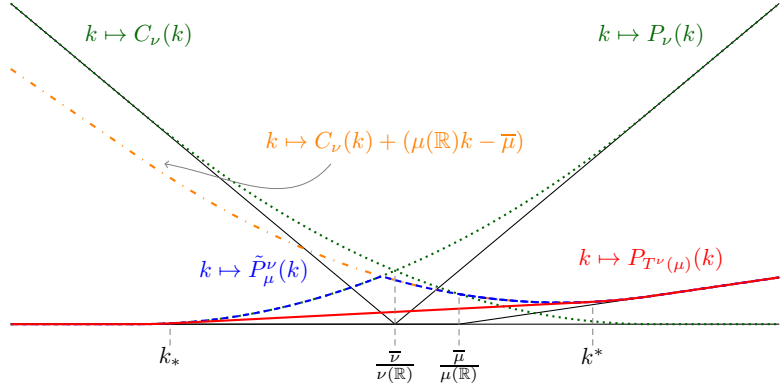

\begin{lem}\label{lem:order}
	Suppose $\mu, \nu \in \sM$. The following are equivalent:
	\begin{itemize}
		\item[(i)] $\mu \leq_E \nu$;
		\item[(ii)] there exists $\eta \in \sM$ such that $\mu \leq_{cx} \eta \leq \nu$;
		\item[(iii)] there exists $\chi \in \sM$ such that $\mu \leq \chi \leq_{cx} \nu$.
	\end{itemize}
\end{lem}

\begin{proof}
	
	That (i) implies (ii) follows from \cref{thm:maximal} (see also the paragraph below \cref{lem:max}). (The {\em shadow} of $\mu$ in $\nu$, see Definition~\ref{defn:shadow} and Theorem~\ref{thm:shadow_potential} below, is another measure that satisfies (ii).)
	
	For (ii) implies (iii), let $\eta$ be as in part (ii). Then it is easily verified that $\chi$, defined by $\chi=\nu-\eta+\mu$, satisfies $\chi \in \sM$ and $\mu\leq\chi\leq_{cx}\nu.$ (To prove (iii) implies (ii) reverse the roles of $\eta$ and $\chi$, i.e., take $\chi$ as in part (iii) and define $\eta=\nu-\chi+\mu$.)
	
	Finally we show that (iii) implies (i). Suppose $\mu \leq \chi \leq_{cx} \nu$. Then, for any non-negative and convex $f:\R\mapsto\R_+$,
	\begin{equation*}
	\int fd\mu \leq\int fd\chi \leq \int fd\nu,
	\end{equation*}
	and thus $\mu\leq_E\nu$. (The proof of (ii) implies (i) is identical.)
\end{proof}
\begin{cor}\label{corE}
	Let $\mu,\nu\in\sM$ with $\mu\leq_E\nu$. Define $\beta\in\sM$ by $\beta:=(\nu(\R)-\mu(\R))\delta_{ \frac{\overline\nu-\overline\mu}{\nu(\R)-\mu(\R)}}$. Then
	$$
	(\nu-\mu)\geq_{ cx} \beta.
	$$
	It follows that $(P_\nu-P_\mu)\geq P_\beta$ where $P_\beta(k) = ((\nu(\R)-\mu(\R))k - (\overline\nu- \overline\mu))^+$.
\end{cor}\begin{proof}
	Let $\eta\in\sM$ be as in (ii) of \cref{lem:order}. Since $\mu\leq_{cx}\eta$, we have that $\mu(\R)=\eta(\R)$ and $\overline\mu=\overline\eta$, and therefore $(\nu-\eta)(\R)=\beta(\R)$ and $\overline{(\nu-\eta)}= \overline{\nu} - \overline{\mu}=\overline\beta$. Since a point mass is smaller in convex order than any other distribution with the same mass and mean, it follows that $\beta\leq_{cx}\nu-\eta\leq_{cx} \nu-\mu$.
\end{proof}

Let $\mu=\mu_1+\mu_2$ for some $\mu_1,\mu_2\in\sM$ and $\nu\in\sM$ with $\mu\leq_E\nu$. Then $\sM^\nu_{\mu_1}\neq\emptyset$ and, in particular, we can embed $\mu_1$ into $\nu$ using any martingale coupling $\pi\in\Pi_M(\mu_1,T^\nu(\mu_1))$. A natural question is then whether $\sM_{\mu_2}^{\nu-T^\nu(\mu_1)}$ is non-empty, so that the remaining mass $\mu_2$ can also be embedded in what remains of $\nu$.
\begin{eg}\label{eg1}
	Let $\mu=\frac{1}{2}(\delta_{-1}+\delta_1)$ and $\nu=\frac{1}{3}(\delta_{-2}+\delta_0+\delta_1)$. Then $\mu\leq_{cx}\nu$. Consider $\mu_1=\frac{2}{3}\mu$ and $\mu_2=\mu-\mu_1=\frac{1}{3}\mu$. Then $T^\nu(\mu_1)=\frac{1}{3}(\delta_{-2}+\delta_2)$. However, $\mu_2\leq_{cx}\nu-T^\nu(\mu_1)$ does not hold. Indeed, $\nu-T^\nu(\mu_1)=\frac{1}{3}\delta_0\leq_{cx}\mu_2$.
\end{eg}
As \cref{eg1} demonstrates, for $\mu_1,\mu_2,\nu\in\sM$ with $\mu_1+\mu_2=\mu\leq_E\nu$,
if we first transport $\mu_1$ to $T^\nu(\mu_1)$, then we cannot,
in general, embed $\mu_2$ in $\nu - T^\nu(\mu_1)$ in a way which respects the martingale property.
As a consequence, for arbitrary measures in convex order we cannot expect the maximal element to \textit{induce} a martingale coupling. In the next section we study the minimal element of $\sM^\nu_\mu$, namely the shadow measure. The shadow measure has the property that if $\mu_1+\mu_2=\mu\leq_E\nu$ and we transport $\mu_1$ to the shadow $S^{\nu}(\mu_1)$ of $\mu_1$ in $\nu$, then $\mu_2$ is in extended convex order with what remains of $\nu$, i.e., $\mu_2 \leq_E \nu - S^\nu(\mu_1)$.

\section{The shadow measure}\label{sec:shadow}

\begin{defn}[Shadow measure]\label{defn:shadow}
	Let $\mu,\nu\in\sM$ and assume $\mu\leq_E\nu$. The shadow of $\mu$ in $\nu$, denoted by $S^\nu(\mu)$, has the following properties
	\begin{enumerate}
		\item\label{s1} $\mu\leq_{cx}S^\nu(\mu)$,
		\item\label{s2} $S^\nu(\mu)\leq\nu$,
		\item \label{s3}If $\eta$ is another measure satisfying $\mu \leq_{cx} \eta\leq \nu$, then $S^\nu(\mu)\leq_{cx}\eta$.
	\end{enumerate}
	\label{lem:shadowDef}
\end{defn}

\begin{rem}\label{rem:Ecx}
	If $\mu \leq_{cx} \nu$ then, in the light of \cref{rem:Ecx0}, $S^{\nu}(\mu) = \nu=T^\nu(\mu)$.
\end{rem}

\begin{prop}[Beiglb\"{o}ck and Juillet \cite{BeiglbockJuillet:16}, Lemma 4.6]
	For $\mu,\nu\in\sM$ with $\mu\leq_E\nu$, $S^\nu(\mu)$ exists and is unique.
\end{prop}




Given $\mu$ and $\nu$ with $\mu \leq _E \nu$ (and, by \cref{rem:Ecx}, with $\mu(\R)<\nu(\R)$) our goal in this section is to construct the shadow measure ${S}^\nu(\mu)$. We do this by finding a corresponding (modified) potential function $P_{S^\nu(\mu)}$ (and then $S^\nu(\mu)$ can be identified as the second derivative of $P_{S^\nu(\mu)}$ in the sense of distributions).

\begin{thm}\label{thm:shadow_potential}
	Let $\mu,\nu\in\sM$ with $\mu\leq_{E}\nu$. Then the shadow of $\mu$ in $\nu$ is uniquely defined and given by
	\begin{equation}\label{eq:shadow_potential}
	P_{S^\nu(\mu)}=P_\nu-(P_\nu-P_{\mu})^c. 
	\end{equation}
\end{thm}

\begin{proof}
	Rephrasing \cref{defn:shadow} above for the shadow measure (and splitting the first element into two parts), a function $h$ is the potential of the shadow of $\mu$ in $\nu$ if
	\begin{enumerate}
		\item[0] $h\in \sD(\mu(\R), \overline\mu)$,
		\item[1] $P_\mu \leq  h$,
		\item[2] $P_\nu-h$ is a potential function, i.e., $P_\nu-h\in \sD(\alpha, \beta)$ for some $\alpha\geq 0, \beta\in \R$,
		\item[3] If $p$ is another potential function satisfying properties 0,1,2 then $h \leq p$.
	\end{enumerate}
	Equivalently we can write this as
	\begin{enumerate}
		\item[0] $h\in \sD(\mu(\R), \overline\mu)$,
		\item[$1'$] $(P_\nu-h) \leq (P_\nu- P_\mu)$,
		\item[2] $P_\nu-h$ is a potential function, i.e., $P_\nu-h\in \sD( \alpha, \beta)$ for some $\alpha\geq 0, \beta\in \R$.
		\item[$3'$] If $p$ is another potential function with properties 0,$1'$,2 then $(P_\nu-h) \geq (P_\nu-p)$.
	\end{enumerate}
	By \cref{cor:CSpotential} with $g=P_\nu$ and $f= P_\mu$ we have $(P_\nu-P_\mu)^c\in\sD(\nu(\R)-\mu(\R),\overline\nu-\overline\mu)$. Now
	set $h=P_\nu-(P_\nu-P_\mu)^c$. First we verify that $h\in \sD(\mu(\R), \overline\mu)$. By applying \cref{lem:diff}, with $g=P_\nu$ and $f=P_\mu$, we have that $h$ is convex and therefore $h=h^c$. Then, since $h\geq P_{\mu(\R)\delta_{ \overline\mu/\mu(\R)}}\geq0$, applying \cref{CSpotential} with $g=P_\nu$ and $f=(P_\nu-P_\mu)^c$ we conclude that $h\in \sD(\mu(\R), \overline\mu)$.
	
	Now we claim that $ P_\nu - h = (P_\nu-P_\mu)^c$ satisfies the properties $1', 2, 3'$. We already saw that  $(P_\nu-P_\mu)^c\in\sD((\nu(\R)-\mu(\R)),\overline\nu-\overline\mu)$, i.e., $(P_\nu-P_\mu)^c$ is a potential function, and thus property 2 is satisfied. On the other hand, properties $1'$ and $3'$ follow from the definition and the maximality of the convex hull, respectively.
\end{proof}
\begin{figure}\centering
	\begin{tikzpicture}[scale=0.8]
	
	\begin{axis}[%
	width=5.028in,
	height=2.754in,
	at={(1.011in,0.642in)},
	scale only axis,
	xmin=-2,
	xmax=2,
	ymin=-0.3,
	ymax=0.6,
	axis line style={draw=none},
	ticks=none
	]
	\addplot [color=black, line width=1.0pt, forget plot]
	table[row sep=crcr]{%
		-2	-0\\
		-1.91657902485342	-0\\
		-1.81818181818182	-0\\
		-1.72300276879786	-0\\
		-1.63636363636364	-0\\
		-1.54716681534369	-0\\
		-1.45454545454545	-0\\
		-1.35423690138381	-0\\
		-1.27272727272727	-0\\
		-1.17140721081572	-0\\
		-1.09090909090909	-0\\
		-1.00320655232955	-0\\
		-0.909090909090909	-0\\
		-0.821698864119408	-0\\
		-0.727272727272727	-0\\
		-0.625862562722587	-0\\
		-0.545454545454545	-0\\
		-0.447954129729544	-0\\
		-0.363636363636364	-0\\
		-0.267949040379642	-0\\
		-0.181818181818182	-0\\
		-0.0917523442360945	-0\\
		0	-0\\
		0.0909329275051746	-0\\
		0.181818181818182	-0\\
		0.276340803623378	-0\\
		0.363636363636364	-0\\
		0.46183178576623	-0\\
		0.545454545454545	-0\\
		0.634803622956617	-0\\
		0.681038175114672	-0\\
		0.727272727272727	-0\\
		0.750149216380473	-0\\
		0.773025705488219	-0\\
		0.784463950042092	-0\\
		0.790183072319028	-0\\
		0.793042633457496	-0\\
		0.795902194595965	-0\\
		0.797331975165199	-0\\
		0.798046865449816	-0\\
		0.798761755734433	-0\\
		0.799119200876741	-0\\
		0.79947664601905	-0\\
		0.799655368590204	-0\\
		0.799834091161358	-0\\
		0.800012813732513	5.12549300504261e-06\\
		0.800191536303667	7.66145214667446e-05\\
		0.800370258874821	0.000148103549928447\\
		0.800548981445975	0.000219592578390149\\
		0.800906426588284	0.000362570635313553\\
		0.801621316872901	0.000648526749160405\\
		0.803051097442135	0.00122043897685407\\
		0.804480878011369	0.00179235120454768\\
		0.807340439149838	0.002936175659935\\
		0.813059561426774	0.00522382457070956\\
		0.81877868370371	0.00751147348148415\\
		0.83006771187711	0.0120270847508441\\
		0.84135674005051	0.0165426960202041\\
		0.86393479639731	0.0255739185589239\\
		0.909090909090909	0.0436363636363637\\
		1.00007025233224	0.0800281009328947\\
		1.09090909090909	0.116363636363636\\
		1.18844925769827	0.155379703079307\\
		1.27272727272727	0.189090909090909\\
		1.35934851571652	0.223739406286608\\
		1.45454545454545	0.261818181818182\\
		1.54994123691727	0.299976494766908\\
		1.63636363636364	0.334545454545454\\
		1.70498279509592	0.361993118038367\\
		1.81818181818182	0.407272727272727\\
		1.90740209505949	0.442960838023795\\
		2	0.48\\
	};
	\addplot [color=blue, dotted, line width=1.0pt, forget plot]
	table[row sep=crcr]{%
		-2	0\\
		-1.95828951242671	0.000217470596700208\\
		-1.91657902485342	0.000869882386800831\\
		-1.86738042151762	0.00219849407460559\\
		-1.81818181818182	0.00413223140495868\\
		-1.77059229348984	0.00657848697578143\\
		-1.72300276879786	0.00959093326170628\\
		-1.63636363636364	0.0165289256198347\\
		-1.54716681534369	0.0256322366407475\\
		-1.45454545454545	0.0371900826446281\\
		-1.35423690138381	0.0521262474417983\\
		-1.27272727272727	0.0661157024793388\\
		-1.17140721081572	0.0858207512860222\\
		-1.09090909090909	0.103305785123967\\
		-1.00320655232955	0.124199647164843\\
		-0.956148730710229	0.135241716725331\\
		-0.932619819900569	0.140142510501985\\
		-0.909090909090909	0.144628098460432\\
		-0.887242897848034	0.14842145957491\\
		-0.865394886605159	0.151856829202395\\
		-0.843546875362283	0.154934196579808\\
		-0.821698864119408	0.157653549492937\\
		-0.798092329907738	0.160189415691246\\
		-0.774485795696068	0.162307302902322\\
		-0.750879261484398	0.16400723827613\\
		-0.727272727272727	0.165289245257122\\
		-0.701920186135192	0.166200629236389\\
		-0.676567644997657	0.166660684442144\\
		-0.651215103860122	0.166577149924312\\
		-0.625862562722587	0.166042287941816\\
		-0.605760558405577	0.165275538460369\\
		-0.585658554088566	0.164205789816786\\
		-0.565556549771556	0.162833007720519\\
		-0.545454545454545	0.161156970942289\\
		-0.521079441523295	0.158718375225911\\
		-0.496704337592045	0.155839398260542\\
		-0.447954129729544	0.150082747953079\\
		-0.363636363636364	0.141499614744913\\
		-0.267949040379642	0.133974695326544\\
		-0.224883611098912	0.131321746910707\\
		-0.203350896458547	0.130168796119287\\
		-0.197967717798456	0.129898881979581\\
		-0.192584539138364	0.12963601459072\\
		-0.189892949808319	0.12950744480856\\
		-0.187201360478273	0.129380474138188\\
		-0.18585556581325	0.129317847309433\\
		-0.184509771148228	0.129255527097664\\
		-0.183836873815716	0.129224395305755\\
		-0.183163976483205	0.12919343440887\\
		-0.182827527816949	0.129178093180007\\
		-0.182491079150693	0.129163054688346\\
		-0.182322854817565	0.129155193705176\\
		-0.182154630484438	0.12914756217806\\
		-0.18198640615131	0.128664917521026\\
		-0.181818181818182	0.128794489046391\\
		-0.181642271979154	0.128931038316866\\
		-0.181466362140127	0.129066860614731\\
		-0.181114542462072	0.129339066485143\\
		-0.180938632623044	0.129475145519364\\
		-0.180762722784017	0.129611227166013\\
		-0.180586812944989	0.129076459706221\\
		-0.180410903105962	0.129068401541556\\
		-0.180234993266934	0.12906035137767\\
		-0.180059083427907	0.129052732453386\\
		-0.179707263749852	0.129036976822835\\
		-0.179003624393742	0.12900537915547\\
		-0.177596345681522	0.128942589979329\\
		-0.176189066969301	0.128880143303639\\
		-0.173374509544861	0.128757325173834\\
		-0.170559952120421	0.128636542879416\\
		-0.164930837271541	0.128400214474608\\
		-0.15930172242266	0.128172017689958\\
		-0.136785263027138	0.127338706581615\\
		-0.0917523442360945	0.126052445259572\\
		-0.0458761721180472	0.125263316707812\\
		0	0.125000328594824\\
		0.0454664637525873	0.125258368839621\\
		0.0909329275051746	0.126033773611156\\
		0.136375554661678	0.127346215887764\\
		0.181818181818182	0.129132423042244\\
		0.22907949272078	0.131559631673975\\
		0.276340803623378	0.134545502657012\\
		0.363636363636364	0.141529051626606\\
		0.46183178576623	0.151661395328132\\
		0.503643165610388	0.156700513789471\\
		0.524548855532467	0.159092231062697\\
		0.545454545454545	0.161156792628065\\
		0.567791814830063	0.163000837224159\\
		0.590129084205581	0.164469876334686\\
		0.634803622956617	0.166894006137953\\
		0.727272727272727	0.17338863636204\\
		0.81877868370371	0.181908089106931\\
		0.909090909090909	0.192396158006623\\
		1.00007025233224	0.205082750663181\\
		1.09090909090909	0.219669797224893\\
		1.18844925769827	0.237706384939799\\
		1.27272727272727	0.255205316933984\\
		1.35934851571652	0.275043520188001\\
		1.45454545454545	0.299008002598171\\
		1.54994123691727	0.325295420873598\\
		1.63636363636364	0.350500024255046\\
		1.70498279509592	0.373595665740981\\
		1.81818181818182	0.41140499518274\\
		1.90740209505949	0.444032720470319\\
		2	0.479999524168838\\
	};
	\addplot [color=red, line width=1.0pt, forget plot]
	table[row sep=crcr]{%
		-2	0\\
		-1.95828951242671	0.000217470596700208\\
		-1.91657902485342	0.000869882386800831\\
		-1.86738042151762	0.00219849407460559\\
		-1.81818181818182	0.00413223140495868\\
		-1.77059229348984	0.00657848697578143\\
		-1.72300276879786	0.00949982660285597\\
		-1.63636363636364	0.0149147712655168\\
		-1.54716681534369	0.0204895714318846\\
		-1.45454545454545	0.0262784052903436\\
		-1.35423690138381	0.0325476885726322\\
		-1.27272727272727	0.0376420393151704\\
		-1.17140721081572	0.0439745418813163\\
		-1.09090909090909	0.0490056733399972\\
		-1.00320655232955	0.0544870808730611\\
		-0.909090909090909	0.060369307364824\\
		-0.821698864119408	0.0658313090513794\\
		-0.727272727272727	0.0717329413896508\\
		-0.625862562722587	0.0780710753695497\\
		-0.545454545454545	0.0830965754144776\\
		-0.447954129729544	0.0891903501430983\\
		-0.363636363636364	0.0944602094393045\\
		-0.267949040379642	0.10044066591198\\
		-0.181818181818182	0.105823843464131\\
		-0.0917523442360945	0.111452957154454\\
		0	0.117187477488958\\
		0.0909329275051746	0.12287078428832\\
		0.181818181818182	0.128551111513785\\
		0.276340803623378	0.134545502657012\\
		0.363636363636364	0.141505768797686\\
		0.46183178576623	0.150097919000078\\
		0.545454545454545	0.157414953704917\\
		0.634803622956617	0.165233044178906\\
		0.727272727272727	0.173324138612147\\
		0.81877868370371	0.181908089106931\\
		0.909090909090909	0.192396158006623\\
		1.00007025233224	0.205082750663181\\
		1.09090909090909	0.219669797224893\\
		1.18844925769827	0.237706384939799\\
		1.27272727272727	0.255205316933984\\
		1.35934851571652	0.275043520188001\\
		1.45454545454545	0.299008002598171\\
		1.54994123691727	0.325295420873598\\
		1.63636363636364	0.350500024255046\\
		1.70498279509592	0.373595665740981\\
		1.81818181818182	0.41140499518274\\
		1.90740209505949	0.444032720470319\\
		2	0.479999524168838\\
	};
	\draw [gray, thin, densely dashed] (-1.75,0.01) -- (-1.75,-0.05);
	\draw [gray, thin, densely dashed] (0.25,0.13) -- (0.25,-0.05);
	
	\draw [gray, thin, densely dashed] (0.35,0.14) -- (0.35,-0.05);
	\draw [gray, thin, densely dashed] (0.75,0.17) -- (0.75,-0.05);
	
	\draw [gray, thin, densely dashed] (0.8,0) -- (0.8,-0.15);

	\node (C)[red] at (1,0.4) {$k\mapsto (P_\nu-P_\mu)^c(k)$};
	\node (P)[blue] at ( -1.3,0.2) {$k\mapsto (P_{\nu}-P_\mu)(k)$};
	
	\draw [gray,->] (1.4,0.06) to[out=230, in=330] (1,0.05);
	\node (P0)[black] at (1.55,0.1) {slope $\nu(\R)-\mu(\R)$};
	\node (bar)[black] at ( 0.8,-0.2)  {$\frac{\overline\nu-\overline\mu}{\nu(\R)-\mu(\R)}$};
	
	\node (k1)[black] at (-1.75,-0.1) {$l_1$};
	\node (k2)[black] at (0.25,-0.1) {$r_1$};
	\node (k3)[black] at (0.35,-0.1) {$l_2$};
	\node (k4)[black] at (0.75,-0.1) {$r_2$};
	
	\end{axis}
	\end{tikzpicture}
	\begin{tikzpicture}[scale=0.8]
	
	\begin{axis}[%
	width=5.028in,
	height=2.754in,
	at={(1.011in,0.642in)},
	scale only axis,
	xmin=-2,
	xmax=2,
	ymin=-0.5,
	ymax=1,
	axis line style={draw=none},
	ticks=none
	]
	\addplot [color=black, line width=1.0pt, forget plot]
	table[row sep=crcr]{%
		-2	-0\\
		-1.91657902485342	-0\\
		-1.81818181818182	-0\\
		-1.72300276879786	-0\\
		-1.63636363636364	-0\\
		-1.54716681534369	-0\\
		-1.45454545454545	-0\\
		-1.35423690138381	-0\\
		-1.27272727272727	-0\\
		-1.17140721081572	-0\\
		-1.09090909090909	-0\\
		-1.00320655232955	-0\\
		-0.909090909090909	-0\\
		-0.821698864119408	-0\\
		-0.727272727272727	-0\\
		-0.625862562722587	-0\\
		-0.545454545454545	-0\\
		-0.447954129729544	-0\\
		-0.363636363636364	-0\\
		-0.267949040379642	-0\\
		-0.181818181818182	-0\\
		-0.0917523442360945	-0\\
		-0.0458761721180472	-0\\
		-0.0229380860590236	-0\\
		-0.0114690430295118	-0\\
		-0.0057345215147559	-0\\
		-0.00286726075737795	-0\\
		-0.00143363037868898	-0\\
		-0.000716815189344488	-0\\
		-0.000358407594672244	-0\\
		-0.000179203797336122	-0\\
		0	0\\
		0.000177603374033544	3.55206748067088e-05\\
		0.000355206748067088	7.10413496134176e-05\\
		0.000710413496134176	0.000142082699226835\\
		0.00142082699226835	0.000284165398453671\\
		0.00284165398453671	0.000568330796907341\\
		0.00568330796907341	0.00113666159381468\\
		0.0113666159381468	0.00227332318762936\\
		0.0227332318762936	0.00454664637525873\\
		0.0454664637525873	0.00909329275051746\\
		0.0909329275051746	0.0181865855010349\\
		0.181818181818182	0.0363636363636363\\
		0.276340803623378	0.0552681607246757\\
		0.363636363636364	0.0727272727272728\\
		0.46183178576623	0.0923663571532461\\
		0.545454545454545	0.109090909090909\\
		0.634803622956617	0.126960724591323\\
		0.727272727272727	0.145454545454545\\
		0.81877868370371	0.163755736740742\\
		0.909090909090909	0.181818181818182\\
		1.00007025233224	0.200014050466447\\
		1.09090909090909	0.218181818181818\\
		1.18844925769827	0.237689851539653\\
		1.27272727272727	0.254545454545455\\
		1.35934851571652	0.271869703143304\\
		1.45454545454545	0.290909090909091\\
		1.54994123691727	0.309988247383454\\
		1.63636363636364	0.327272727272727\\
		1.70498279509592	0.340996559019184\\
		1.81818181818182	0.363636363636364\\
		1.90740209505949	0.381480419011898\\
		2	0.4\\
	};
	\addplot [color=blue, dotted, line width=1.0pt, forget plot]
	table[row sep=crcr]{%
		-2	0\\
		-1.91657902485342	0.000869882386800831\\
		-1.81818181818182	0.00413223140495868\\
		-1.72300276879786	0.00959093326170628\\
		-1.63636363636364	0.0165289256198347\\
		-1.54716681534369	0.0256322366407475\\
		-1.45454545454545	0.0371900826446281\\
		-1.35423690138381	0.0521262474417983\\
		-1.27272727272727	0.0661157024793388\\
		-1.17140721081572	0.0858207512860222\\
		-1.09090909090909	0.103305785123967\\
		-1.00320655232955	0.124199647164843\\
		-0.909090909090909	0.148760330578512\\
		-0.821698864119408	0.170483664480881\\
		-0.774485795696068	0.179858728517657\\
		-0.727272727272727	0.187561982454777\\
		-0.676567644997657	0.193973140089941\\
		-0.625862562722587	0.198456054625072\\
		-0.545454545454545	0.201611568195798\\
		-0.447954129729544	0.200287384861301\\
		-0.363636363636364	0.200165366844466\\
		-0.267949040379642	0.202179589751035\\
		-0.181818181818182	0.205950142018082\\
		-0.0917523442360945	0.211877121893853\\
		0	0.219999872829605\\
		0.0909329275051746	0.230127031373444\\
		0.181818181818182	0.242314152566714\\
		0.276340803623378	0.257179558256049\\
		0.363636363636364	0.272892579123249\\
		0.46183178576623	0.292844369134017\\
		0.545454545454545	0.310702333919897\\
		0.634803622956617	0.324766246801399\\
		0.727272727272727	0.333016590517152\\
		0.81877868370371	0.334867788279656\\
		0.909090909090909	0.330578337483929\\
		1.00007025233224	0.324996348816508\\
		1.09090909090909	0.321487238867364\\
		1.18844925769827	0.320017007662797\\
		1.27272727272727	0.320660890763356\\
		1.35934851571652	0.323173747926235\\
		1.45454545454545	0.328099171797474\\
		1.54994123691727	0.335307124671952\\
		1.63636363636364	0.343227488514124\\
		1.70498279509592	0.351875904071431\\
		1.81818181818182	0.367768996580196\\
		1.90740209505949	0.382551389680759\\
		2	0.399999411241802\\
	};
	\addplot [color=red, line width=1.0pt, forget plot]
	table[row sep=crcr]{%
		-2	0\\
		-1.91657902485342	0.000869882386800831\\
		-1.81818181818182	0.00413223140495868\\
		-1.72300276879786	0.00959093326170628\\
		-1.63636363636364	0.0165289256198347\\
		-1.54716681534369	0.0252833248897676\\
		-1.45454545454545	0.0345454722316835\\
		-1.35423690138381	0.0445763397446485\\
		-1.27272727272727	0.0527273125212879\\
		-1.17140721081572	0.0628593310322352\\
		-1.09090909090909	0.0709091528108923\\
		-1.00320655232955	0.0796794173328458\\
		-0.909090909090909	0.0890909931004967\\
		-0.821698864119408	0.0978302082238923\\
		-0.727272727272727	0.107272833390101\\
		-0.625862562722587	0.117413862175863\\
		-0.545454545454545	0.125454673679705\\
		-0.447954129729544	0.135204727107557\\
		-0.363636363636364	0.14363651396931\\
		-0.267949040379642	0.153205257929874\\
		-0.181818181818182	0.161818354258914\\
		-0.0917523442360945	0.170824948968483\\
		0	0.180000194548519\\
		0.0909329275051746	0.189093498355828\\
		0.181818181818182	0.198182034838123\\
		0.276340803623378	0.207634308511915\\
		0.363636363636364	0.216363875127727\\
		0.46183178576623	0.226183429280573\\
		0.545454545454545	0.234545715417332\\
		0.634803622956617	0.243480634031746\\
		0.727272727272727	0.252727555706936\\
		0.81877868370371	0.261878162476502\\
		0.909090909090909	0.270909395996541\\
		1.00007025233224	0.280007341383109\\
		1.09090909090909	0.289091236286145\\
		1.18844925769827	0.298845264825247\\
		1.27272727272727	0.307273076575749\\
		1.35934851571652	0.315935211407196\\
		1.45454545454545	0.325454916865354\\
		1.54994123691727	0.334994506701978\\
		1.63636363636364	0.343227488514124\\
		1.70498279509592	0.351875904071431\\
		1.81818181818182	0.367768996580196\\
		1.90740209505949	0.382551389680759\\
		2	0.399999411241802\\
	};

	\draw [gray, thin, densely dashed] (-1.6,0.02) -- (-1.6,-0.1);
	\draw [gray, thin, densely dashed] (1.6,0.33) -- (1.6,-0.1);
	
	\draw [gray, thin, densely dashed] (0,0) -- (0,-0.1);
	
	\node (C)[red] at (1.4,0.5) {$k\mapsto (P_\nu-P_\mu)^c(k)$};
	\node (P)[blue] at ( -1,0.33) {$k\mapsto (P_{\nu}-P_\mu)(k)$};
	
	\draw [gray,->] (0.8,-0.01) to[out=230, in=350] (0.4,0.025);
	\node (P0)[black] at (1.1,0.05) {slope $\nu(\R)-\mu(\R)$};
	\node (bar)[black] at ( 0,-0.2)  {$\frac{\overline\nu-\overline\mu}{\nu(\R)-\mu(\R)}$};
	
	\node (k1)[black] at (-1.6,-0.2) {$l$};
	\node (k2)[black] at (1.6,-0.2) {$r$};

	\end{axis}
	\end{tikzpicture}
	\begin{tikzpicture}[scale=0.8]
	
	\begin{axis}[%
	width=5.028in,
	height=2.754in,
	at={(1.011in,0.642in)},
	scale only axis,
	xmin=-2,
	xmax=2,
	ymin=-0.3,
	ymax=1,
	axis line style={draw=none},
	ticks=none
	]
	\addplot [color=black, line width=1.0pt, forget plot]
	table[row sep=crcr]{%
		-2	-0\\
		-1.91657902485342	-0\\
		-1.81818181818182	-0\\
		-1.72300276879786	-0\\
		-1.63636363636364	-0\\
		-1.54716681534369	-0\\
		-1.45454545454545	-0\\
		-1.35423690138381	-0\\
		-1.27272727272727	-0\\
		-1.17140721081572	-0\\
		-1.09090909090909	-0\\
		-1.00320655232955	-0\\
		-0.909090909090909	-0\\
		-0.821698864119408	-0\\
		-0.727272727272727	-0\\
		-0.625862562722587	-0\\
		-0.545454545454545	-0\\
		-0.447954129729544	-0\\
		-0.363636363636364	-0\\
		-0.267949040379642	-0\\
		-0.181818181818182	-0\\
		-0.0917523442360945	-0\\
		-0.0458761721180472	-0\\
		-0.0229380860590236	-0\\
		-0.0114690430295118	-0\\
		-0.0057345215147559	-0\\
		-0.00286726075737795	-0\\
		-0.00143363037868898	-0\\
		-0.000716815189344488	-0\\
		-0.000358407594672244	-0\\
		-0.000179203797336122	-0\\
		0	0\\
		0.000177603374033544	8.8801687016772e-05\\
		0.000355206748067088	0.000177603374033544\\
		0.000710413496134176	0.000355206748067088\\
		0.00142082699226835	0.000710413496134176\\
		0.00284165398453671	0.00142082699226835\\
		0.00568330796907341	0.00284165398453671\\
		0.0113666159381468	0.00568330796907341\\
		0.0227332318762936	0.0113666159381468\\
		0.0454664637525873	0.0227332318762936\\
		0.0909329275051746	0.0454664637525873\\
		0.181818181818182	0.0909090909090908\\
		0.276340803623378	0.138170401811689\\
		0.363636363636364	0.181818181818182\\
		0.46183178576623	0.230915892883115\\
		0.545454545454545	0.272727272727273\\
		0.634803622956617	0.317401811478308\\
		0.727272727272727	0.363636363636364\\
		0.81877868370371	0.409389341851855\\
		0.909090909090909	0.454545454545455\\
		1.00007025233224	0.500035126166118\\
		1.09090909090909	0.545454545454545\\
		1.18844925769827	0.594224628849133\\
		1.27272727272727	0.636363636363636\\
		1.35934851571652	0.679674257858261\\
		1.45454545454545	0.727272727272727\\
		1.54994123691727	0.774970618458634\\
		1.63636363636364	0.818181818181818\\
		1.70498279509592	0.852491397547959\\
		1.81818181818182	0.909090909090909\\
		1.90740209505949	0.953701047529744\\
		2	1\\
	};
	\addplot [color=blue, dotted, line width=1.0pt, forget plot]
	table[row sep=crcr]{%
		-2	0\\
		-1.91657902485342	0.000869882386800831\\
		-1.81818181818182	0.00413223140495868\\
		-1.72300276879786	0.00959093326170628\\
		-1.63636363636364	0.0165289256198347\\
		-1.54716681534369	0.0256322366407475\\
		-1.45454545454545	0.0371900826446281\\
		-1.35423690138381	0.0521262474417983\\
		-1.27272727272727	0.0661157024793388\\
		-1.17140721081572	0.0858207512860222\\
		-1.09090909090909	0.103305785123967\\
		-1.00320655232955	0.124199647164843\\
		-0.909090909090909	0.146694214519472\\
		-0.821698864119408	0.165601372672562\\
		-0.727272727272727	0.183884292050048\\
		-0.625862562722587	0.201037000003866\\
		-0.545454545454545	0.212809890429822\\
		-0.447954129729544	0.225594272940513\\
		-0.363636363636364	0.238105179273283\\
		-0.267949040379642	0.254487380583385\\
		-0.181818181818182	0.271008814771129\\
		-0.0917523442360945	0.290614292362813\\
		0	0.312500164297412\\
		0.0909329275051746	0.336266918513413\\
		0.181818181818182	0.362086872678147\\
		0.276340803623378	0.391130717218551\\
		0.363636363636364	0.419938079532311\\
		0.46183178576623	0.454619181502125\\
		0.545454545454545	0.485537073999983\\
		0.634803622956617	0.517031062272142\\
		0.727272727272727	0.547520764209887\\
		0.81877868370371	0.575582601804173\\
		0.909090909090909	0.601239402044244\\
		1.00007025233224	0.625017727202022\\
		1.09090909090909	0.648760514969866\\
		1.18844925769827	0.676551023994088\\
		1.27272727272727	0.702478734481946\\
		1.35934851571652	0.730978647536514\\
		1.45454545454545	0.764462685482679\\
		1.54994123691727	0.800289624489699\\
		1.63636363636364	0.834438239764801\\
		1.70498279509592	0.863616113618478\\
		1.81818181818182	0.913391828629002\\
		1.90740209505949	0.954772900225842\\
		2	0.99999973064989\\
	};
	\addplot [color=red, line width=1.0pt, forget plot]
	table[row sep=crcr]{%
		-2	0\\
		-1.91657902485342	0.000869882386800831\\
		-1.81818181818182	0.00413223140495868\\
		-1.72300276879786	0.00959093326170628\\
		-1.63636363636364	0.0165289256198347\\
		-1.54716681534369	0.0256322366407475\\
		-1.45454545454545	0.0371900826446281\\
		-1.35423690138381	0.0521262474417983\\
		-1.27272727272727	0.0661157024793388\\
		-1.17140721081572	0.0850486503118394\\
		-1.09090909090909	0.100142050190845\\
		-1.00320655232955	0.116586278785503\\
		-0.909090909090909	0.134232964694667\\
		-0.821698864119408	0.150618975728574\\
		-0.727272727272727	0.168323879198489\\
		-0.625862562722587	0.187338288070725\\
		-0.545454545454545	0.202414793702312\\
		-0.447954129729544	0.220696124553436\\
		-0.363636363636364	0.236505708206134\\
		-0.267949040379642	0.254447084165479\\
		-0.181818181818182	0.271008814771129\\
		-0.0917523442360945	0.290614292362813\\
		0	0.312500164297412\\
		0.0909329275051746	0.336266918513413\\
		0.181818181818182	0.362086872678147\\
		0.276340803623378	0.391043991837767\\
		0.363636363636364	0.418323866169495\\
		0.46183178576623	0.449009948889569\\
		0.545454545454545	0.47514207262221\\
		0.634803622956617	0.503063671447508\\
		0.727272727272727	0.531960279074926\\
		0.81877868370371	0.560555902857744\\
		0.909090909090909	0.588778485527641\\
		1.00007025233224	0.617209542617342\\
		1.09090909090909	0.645596691980357\\
		1.18844925769827	0.676078007317685\\
		1.27272727272727	0.702478734481946\\
		1.35934851571652	0.730978647536514\\
		1.45454545454545	0.764462685482679\\
		1.54994123691727	0.800289624489699\\
		1.63636363636364	0.834438239764801\\
		1.70498279509592	0.863616113618478\\
		1.81818181818182	0.913391828629002\\
		1.90740209505949	0.954772900225842\\
		2	0.99999973064989\\
	};
	\draw [gray, thin, densely dashed] (-1.25,0.07) -- (-1.25,-0.05);
	\draw [gray, thin, densely dashed] (-0.25,0.26) -- (-0.25,-0.05);
	
	\draw [gray, thin, densely dashed] (0.25,0.38) -- (0.25,-0.05);
	\draw [gray, thin, densely dashed] (1.25,0.69) -- (1.25,0.25);
	
	\draw [gray, thin, densely dashed] (1.25,0.07) -- (1.25,-0.05);

	\draw [gray, thin, densely dashed] (0,0) -- (0,-0.15);

	\node (C)[red] at (0.9,0.8) {$k\mapsto (P_\nu-P_\mu)^c(k)$};
	\node (P)[blue] at ( -1.3,0.3) {$k\mapsto (P_{\nu}-P_\mu)(k)$};
	
	\draw [gray,->] (1,0.1) to[out=230, in=330] (0.4,0.15);
	\node (P0)[black] at (1.1,0.15) {slope $\nu(\R)-\mu(\R)$};
	\node (bar)[black] at ( 0,-0.2)  {$\frac{\overline\nu-\overline\mu}{\nu(\R)-\mu(\R)}$};
	
	\node (k1)[black] at (-1.25,-0.1) {$l_1$};
	\node (k2)[black] at (-0.25,-0.1) {$r_1$};
	\node (k3)[black] at (0.25,-0.1) {$l_2$};
	\node (k4)[black] at (1.25,-0.1) {$r_2$};
	
	\end{axis}
	\end{tikzpicture}
	\caption{In each drawing, the dotted curve corresponds to $P_\nu-P_\mu$, while the solid curve below it corresponds to $P_\nu-P_{S^\nu(\mu)}=(P_\nu-P_\mu)^c$.
		Top and bottom figures: $(P_\nu-P_\mu)^c$ is linear (and lies strictly below $P_\nu-P_\mu$) on $(l_1,r_1)\cup(l_2,r_2)$ and $(P_\nu-P_\mu)^c=P_\nu-P_\mu$ on $\R\setminus((l_1,r_1)\cup(l_2,r_2))$.
		Middle figure: $(P_\nu-P_\mu)^c$ is linear (and lies strictly below $P_\nu-P_\mu$) on $(l,r)$ and $(P_\nu-P_\mu)^c=P_\nu-P_\mu$ on $\R\setminus(l,r)$.}
	\label{fig:LC}
\end{figure}
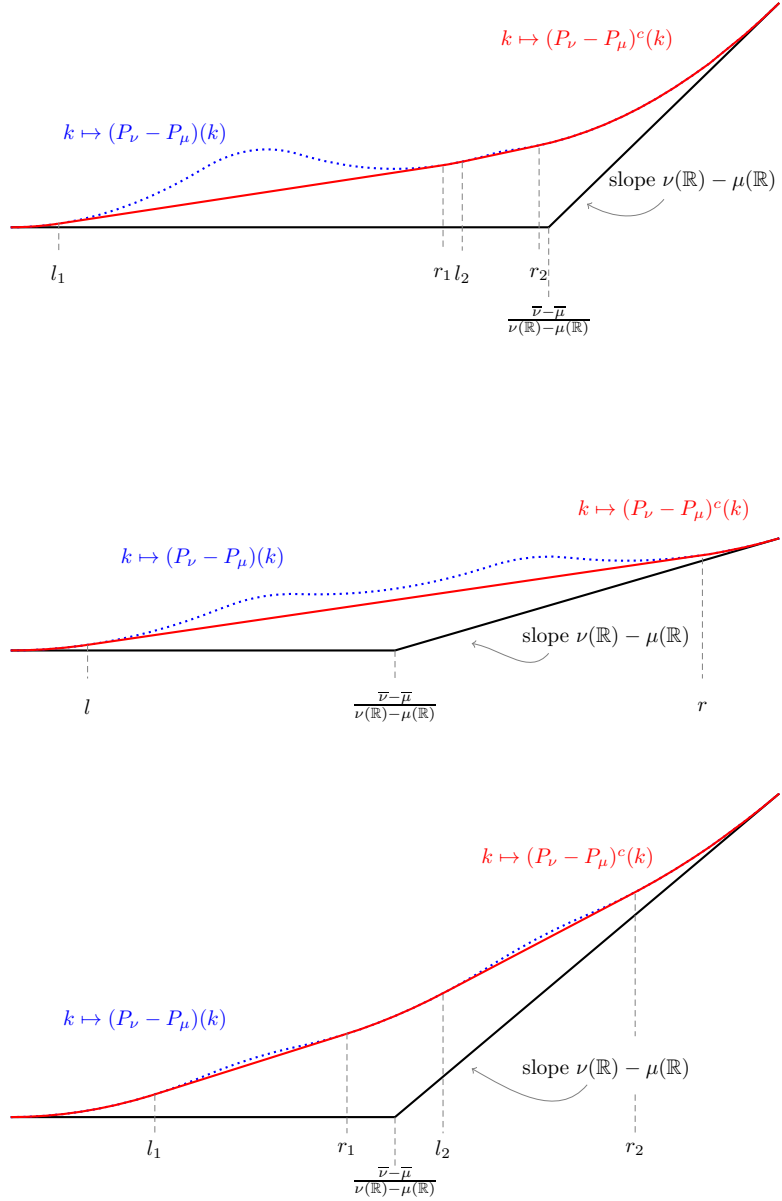
\begin{eg}\label{eg:shadow}
	Let $\hat\mu,\nu\in\sM$ (with $\hat\mu\leq_{cx}\nu$) be given by
	$$
	\hat\mu=\frac{U[-1,-0.5]+U[0.5,1]}{2}\quad\textrm{and}\quad \nu=U[-2,2],
	$$
	and let $\mu\in\sM$ with $\mu\leq\hat\mu$ (so that $\mu\leq_E\nu$). For three different choices of $\mu$ (that are related to the parametrisations of $\hat\mu$ corresponding to the left-curtain ($\pi^{lc}$), middle-curtain ($\pi^{mc}$) and sunset ($\pi^{sun}$) couplings, respectively) we find the shadow $S^\nu(\mu)$.
	\begin{enumerate}
		\item (Left-curtain) Set $\mu=\hat\mu\lvert_{[-1,0.6]}$. Then $\nu(\R)-\mu(\R)=0.4$ and $\overline\nu-\overline\mu=0.32$. In particular, $S^\nu(\mu)=\nu\lvert_{[l_1,r_1]\cup[l_2,r_2]}$, where $l_1=-1.75$, $r_1=0.25$, $l_2=0.35$ and $r_2=0.75$. See the top drawing in \cref{fig:LC}.
		\item (Middle-curtain) Set $u=4/5$ and $\mu=S^\mu(u\delta_{\overline{\hat\mu}})$. Then $\mu=\hat\mu\lvert_{[-0.9,0.9]}$,  $\nu(\R)-\mu(\R)=0.2$ and $\overline\nu-\overline\mu=0$. In particular, $S^\nu(\mu)=\nu\lvert_{[l,r]}$, where $-l=r=1.6$. See the middle drawing in \cref{fig:LC}.
		\item (Sunset) Fix $u\in(0,1)$ and set $\mu=u\hat\mu$. If $u\leq0.25$, then $\mu\leq\nu$ and therefore $S^\nu(\mu)=\mu$. Consider $u=0.5$. Then $\nu(\R)-\mu(\R)=0.5$ and $\overline\nu-\overline\mu=0$. In particular, $S^\nu(\mu)=\nu\lvert_{[l_1,r_1]\cup[l_2,r_2]}$, where $l_1=-1.25$, $r_1=-0.25$, $l_2=0.25$ and $r_2=1.25$. See the bottom drawing in \cref{fig:LC}.
	\end{enumerate}
\end{eg}

We now turn to the associativity of the shadow measure. As alluded to in the introduction, it is one of the most important results on the structure of shadows. The proof of the associativity (\cref{thm:shadow_assoc}) given in Beiglb\"{o}ck and Juillet \cite{BeiglbockJuillet:16} is delicate and based on the approximation of $\mu$ by atomic measures. Thanks to \cref{thm:shadow_potential}, we are able to provide a simple proof of \cref{thm:shadow_assoc}.
\begin{thm}[Beiglb\"{o}ck and Juillet \cite{BeiglbockJuillet:16}, Theorem 4.8]
	\label{thm:shadow_assoc}
	Suppose $\mu=\mu_1+\mu_2$ for some $\mu_1,\mu_2\in\sM$ and $\mu \leq_E \nu$. Then $\mu_2\leq_E\nu-S^\nu(\mu_1)$ and
	\begin{equation}\label{eq:assoc}
	S^\nu(\mu_1+\mu_2)=S^\nu(\mu_1)+S^{\nu-S^\nu(\mu_1)}(\mu_2).
	\end{equation}
\end{thm}

\begin{proof}
	
	
	

	We first prove that  $\mu_2\leq_E\nu-S^\nu(\mu_1)$. Define $P_\theta:\R\to\R_+$ by
	$$
	P_\theta(k)=(P_\nu-P_{\mu_1})^c(k)-((P_\nu-P_{\mu_1})^c-P_{\mu_2})^c(k),\quad k\in\R.
	$$
	We will show that $P_\theta\in\sC(\mu_2,\nu-S^\nu(\mu_1))$. Then the second derivative of $P_\theta$ corresponds to a measure $\theta\in\sM^{\nu-S^\nu(\mu_1)}_{\mu_2}$, which by Lemma~\ref{lem:order} is enough to prove the assertion.
	
	Convexity of $P_\theta$ is a direct consequence of \cref{lem:diff} with $g=(P_\nu-P_{\mu_1})^c$ and $f=P_{\mu_2}$. Moreover, since
	$$
	P_{\nu-S^\nu(\mu_1)}=P_\nu-P_{S^\nu(\mu_1)}=(P_\nu-P_{\mu_1})^c,
	$$
	we have that
	$$
	P_{\nu-S^{\nu}(\mu_1)}-P_\theta =((P_\nu-P_{\mu_1})^c-P_{\mu_2})^c\leq(P_\nu-P_{\mu_1})^c-P_{\mu_2},
	$$
	and it follows that $(P_{\nu-S^{\nu}(\mu_1)}-P_\theta)$ is convex and $P_{\mu_2}\leq P_\theta$. To prove that $\mu_2 \leq_{E} \nu - S^{\nu}(\mu_1)$ it only remains to show that $P_\theta$ has the correct limiting behaviour to ensure that $P_\theta \in \sD(\mu_2(\R),\overline{\mu_2})$. For this we will apply \cref{CSpotential} (together with \cref{corE}) to each of the convex hulls in the definition of $P_\theta$ and then to $P_\theta$ itself.
	
	First, since $\mu_1\leq_E\nu$, by \cref{cor:CSpotential}
	we have that $(P_\nu-P_{\mu_1})^c\in\sD(\nu(\R)-\mu_1(\R),\overline\nu-\overline{\mu_1})$. Similarly, since $\mu_1+\mu_2\leq_E\nu$, $(P_\nu-P_{\mu_1}-P_{\mu_2})^c\in\sD(\nu(\R)-\mu_1(\R)-\mu_2(\R),\overline{\nu-\mu_1-\mu_2})$. But, by \cref{lem:equal_hulls}, with $f=(P_\nu-P_{\mu_1})$ and $ g=P_{\mu_2}$, we have that $((P_\nu-P_{\mu_1})^c-P_{\mu_2})^c=(P_\nu-P_{\mu_1}-P_{\mu_2})^c$. Finally, recall that $P_\theta\geq P_{\mu_2}\geq P_{\mu_2(\R)\delta_{\overline{\mu_2}/ \mu_2(\R)}}$ and, since $P_\theta$ is convex, $P_\theta=P_\theta^c$. Therefore, by applying \cref{CSpotential} with $g=(P_\nu-P_{\mu_1})^c$ and $f=((P_\nu-P_{\mu_1})^c-P_{\mu_2})^c$, we conclude that $P_\theta\in\sD(\mu_2(\R),\overline{\mu_2})$.

	We are left to prove the associativity property \eqref{eq:assoc}. By applying \cref{thm:shadow_potential} to $S^\nu(\mu_1+\mu_2)$ 
	we have by \cref{lem:equal_hulls} that
	$$
	P_{S^\nu(\mu_1+\mu_2)}= P_\nu - \left((P_\nu-P_{\mu_1})-P_{\mu_2}\right)^c= P_\nu - \left((P_\nu-P_{\mu_1})^c-P_{\mu_2}\right)^c,
	$$
	whilst applying \cref{thm:shadow_potential} to $S^\nu(\mu_1)$ and $S^{\nu-S^\nu(\mu_1)}(\mu_2)$ gives
	\begin{align*}
	P_{S^\nu(\mu_1)}+P_{S^{\nu-S^\nu(\mu_1)}(\mu_2)}&=\left\{P_\nu-(P_\nu-P_{\mu_1})^c\right\}+\left\{P_{\nu-S^\nu(\mu_1)}-(P_{\nu-S^\nu(\mu_1)}-P_{\mu_2})^c\right\}\\
	&=P_\nu-\left((P_\nu-P_{\mu_1})^c-P_{\mu_2}\right)^c,
	\end{align*}
	where we again used that $P_{\nu-S^\nu(\mu_1)}=P_\nu - P_{S^\nu(\mu_1)}=(P_\nu-P_{\mu_1})^c$.
\end{proof}

We give one further result which is easy to prove using \cref{thm:shadow_potential} and which describes a
structural property of the shadow.
\begin{prop}\label{prop:shadowSum}
	Suppose $\xi,\mu,\nu\in\sM$ with $\xi \leq \mu \leq_E \nu$. Then, $\xi \leq_E \nu$, $\xi \leq_E S^\nu(\mu)$ and
	$$
	S^{S^\nu(\mu)}(\xi) = S^\nu(\xi).
	$$
\end{prop}
\begin{proof}
	Let $f$ be non-negative and convex. Then, $\int f d \xi \leq \int f d \mu \leq \int f dS^\nu(\mu) \leq \int f d \nu$ and both $\xi  \leq_E S^\nu(\mu)$ and $\xi \leq_E \nu$.
	Then, by \cref{thm:shadow_potential}, we have $P_{S^\nu(\xi)}=P_\nu-(P_\nu-P_{\xi})^c$. Applying \cref{thm:shadow_potential} to $P_{S^{S^\nu(\mu)}(\xi)}$, and writing $\sigma = \mu - \xi$ so that $P_\mu = P_\xi + P_\sigma$,
	\[ P_{S^{S^\nu(\mu)}(\xi)}
	=P_{S^\nu(\mu)}-(P_{S^\nu(\mu)}-P_{\xi})^c=\Big\{P_\nu-(P_\nu-P_{\xi}-P_{\sigma})^c\Big\}-\Big(P_\nu-(P_\nu-P_{\xi}-P_{\sigma})^c-P_{\xi}\Big)^c.
	\]
	Three applications of \cref{lem:equal_hulls} give that
	\[P_{S^{S^\nu(\mu)}(\xi)}=\Big\{P_\nu-((P_\nu-P_{\xi})^c-P_{\sigma})^c\Big\}-\Big((P_\nu-P_{\xi})^c-((P_\nu-P_{\xi})^c-P_{\sigma})^c\Big)^c. \]
	Finally, by \cref{lem:diff} (with $g=(P_\nu-P_{\xi})^c$ and $f=P_{\sigma}$) we have that $(P_\nu-P_{\xi})^c-((P_\nu-P_{\xi})^c-P_{\sigma})^c$ is convex and then
	\begin{align*}
	P_{S^{S^\nu(\mu)}(\xi)}
	& = \Big\{P_\nu-((P_\nu-P_{\xi})^c-P_{\sigma})^c\Big\}-\Big((P_\nu-P_{\xi})^c-((P_\nu-P_{\xi})^c-P_{\sigma})^c\Big) \\
	&=  P_\nu - (P_\nu - P_\xi)^c
	=P_{S^\nu(\xi)}.
	\end{align*}
\end{proof}

\begin{eg}\label{eg2}
	The assertion of \cref{prop:shadowSum} does not hold for $\xi,\mu,\nu\in\sM$ with $\xi\leq_{E}\mu\leq_{E}\nu$. To see this, let $\xi=\frac{1}{3}\delta_0$, $\mu=\frac{1}{3}(\delta_{-2}+\delta_2)$ and $\nu=\frac{1}{3}(\delta_{-2}+\delta_0+\delta_2)$. Then $S^\nu(\mu)=\mu$ and $S^{S^\nu(\mu)}(\xi)=S^\mu(\xi)=\frac{1}{6}(\delta_{-2}+\delta_2)\neq\xi=S^\nu(\xi)$.
\end{eg}

\section{Proofs}\label{sec:proofs}

\begin{proof}[Proof of \cref{lem:linear}]
	If $f^c \equiv -\infty$ then it is linear and we are done. Henceforth we exclude this case.
	
	Suppose $f^c$ is not a straight line on $(a,b)$. Then, by the convexity of $f^c$, for all $x\in(a,b)$ we have
	\[ f^c(x) <\frac{b-x}{b-a} f^c(a) + \frac{x-a}{b-a} f^c(b) = L^{f^c}_{a,b}(x) . \]
	
	Let $\eta = \inf_{u \in (a,b)} \{ f(u) -  L^{f^c}_{a,b}(u) \}$. If $\eta \geq 0$ then $f \geq L^{f^c}_{a,b}$ on $(a,b)$ and $f \geq f^c \vee L^{f^c}_{a,b}$, contradicting the maximality of $f^c$ as a convex minorant of $f$.
	
	Now suppose that $\eta<0$. Since $f$ is lower semi-continuous, $f-L^{f^c}_{a,b}$ is also lower semi-continuous, and therefore attains its infimum on $[a,b]$. Fix $z\in\arginf_{u\in[a,b]}\{f(u) -  L^{f^c}_{a,b}(u)\}$. Since $f(k)-L^{f^c}_{a,b}(k)=f(k)-f^c(k)\geq0$ for $k\in\{a,b\}$, $a,b\notin\arginf_{u\in[a,b]}\{f(u) -  L^{f^c}_{a,b}(u)\}$, and thus $z\in(a,b)$. Then since $f > f^{c}$ on $(a,b)$ we have $0 > \eta = f(z)-L^{f^c}_{a,b}(z)> f^c(z) -  L^{f^c}_{a,b}(z)$.
	Then $f^c \vee (L^{f^c}_{a,b} + \eta)$ is convex, is a minorant of $f$ and is strictly larger than $f^c$ (in particular at $z$) again contradicting the maximality of $f^c$ as a convex minorant of $f$.
\end{proof}
\begin{proof}[Proof of \cref{lem:diff}]
	Let $h=g-f$. We show that $\psi = g - h^c$ is convex.

If $h^c \equiv -\infty$ then $\psi \equiv +\infty$ which is convex. Henceforth we exclude this case.
	
	First note that, since $h^c(y)\leq h(y)$, $\psi(y)=g(y)-h^c(y)\geq f(y)$, $y\in\R$.
	
	Define
	\begin{equation*}
	\sA_=:=\{y:h(y)=h^c(y)\}\quad\textrm{and}\quad\sA_>:=\{y:h(y)>h^c(y)\}.
	\end{equation*}
	Then $\psi=f$ on $\sA_=$, while  $\psi>f$ on $\sA_>$.
	
	Recall that $\phi:\R\mapsto\R$ is convex if for all $x,y,z\in\R$, with $x\leq y\leq z$,
	\begin{equation*}
	\phi(y)\leq L^\phi_{x,z}(y).
	\end{equation*}
	
	Suppose $y\in\sA_=$. Then, for all $x\leq y\leq z$, since $f$ is convex and $\psi \geq f$,
	\begin{equation*}
	\psi(y)=f(y)\leq L^f_{x,z}(y)\leq L^\psi_{x,z}(y).
	\end{equation*}

	In the rest of the proof we take $y\in\sA_>$,  and $x,z \in \R$ with $-\infty < x \leq y \leq z < \infty$ and show that $\psi(y) \leq L^{\psi}_{x,z}(y)$.
	Let $\sB_y$ be the set of open intervals containing $y$ which are subsets of $\sA_>$. If $y \in \sA_>$ then, by continuity of $h$ and $h^c$, $\sB_y:=\{(a,b)\subseteq\R:y\ni(a,b)\subseteq\sA_>\}$ is non-empty. Moreover $\sB_y$ has a largest element: $\hat{\sB}_y:=\sup\sB_y$. Denote by $X(y)$ (resp. $Z(y)$) the left (resp. right) end-point of $\hat{\sB}_y$. By \cref{lem:linear}, we have that $h^c$ is linear on $(X(y),Z(y))$. Moreover, by continuity of $h$ and $h^c$, if $X(y)$ (resp. $Z(y)$) is finite, then $X(y)\in\sA_=$ (resp. $Z(y)\in\sA_=$). If both $X(y)$ and $Z(y)$ are finite then $h^c(y) = L^h_{X(y),Z(y)}(y)$. In general, $\hat{\sB} = (X(y),Z(y)) \subseteq \sA_>$.
	
	Suppose $-\infty < x\leq X(y)<Z(y)\leq z < \infty$.
	Then, since
	$g$ is convex and $h^c(y)=L^h_{X(y),Z(y)}(y)$,
	\[
	\psi(y)=g(y)-h^c(y)
	\leq L^g_{X(y),Z(y)}(y)-L^h_{X(y),Z(y)}(y)
	=L^f_{X(y),Z(y)}(y) \]
	and then
	\[ \psi(y)  \leq L^f_{X(y),Z(y)}(y) \leq L^{L^f_{x,z}}_{X(y),Z(y)}(y)
	=L^f_{x,z}(y)\nonumber\\
	\leq L^\psi_{x,z}(y) \]
	by the convexity of $f$ (and hence $f \leq L^f_{x,z}$ on $[x,z]$) and the fact that $f\leq \psi$.
	
	Suppose $X(y)\leq x\leq y\leq z\leq Z(y)$. Note that we allow $X(y)=-\infty$ (resp. $Z(y)=\infty$), but in that case $x>X(y)=-\infty$ (resp. $z<Z(y)=\infty$). By Lemma~\ref{lem:linear}, $h^c$ is linear on $(x,z)$, and therefore $h^c(y)=L^{h^c}_{x,z}(y)$. Using convexity of $g$ on $\R$ we conclude that
	\[
	\psi(y)=g(y)-h^c(y)
	\leq L^g_{x,z}(y) - L^{h^c}_{x,z}(y)
	=L^\psi_{x,z}(y).
	\]
	
	Suppose $X(y)\leq x\leq y<Z(y)\leq z$. 
	(The case $x\leq X(y)<y\leq z\leq Z(y)$ follows by symmetry.) Since $z<\infty$ we have that $Z(y)<\infty$, but $X(y)$ may be finite or infinite.
	In this case $h^c$ is also linear on $(x,Z(y))$ and therefore $h^c(y)=L^{h^c}_{x,Z(y)}(y)$. Then
	\begin{equation}
	\psi(y)=g(y)-h^c(y) \leq L^g_{x,Z(y)}(y)-L^{h^c}_{x,Z(y)}(y) = L^\psi_{x,Z(y)}(y),
	\label{eq:A}
	\end{equation}
	where the inequality follows from the convexity of $g$ on $\R$. Now note that, since $f$ is convex and $f\leq \psi$,
	$$
	\psi(Z(y))=g(Z(y))-h(Z(y))=f(Z(y))\leq L^f_{x,z}(Z(y))\leq L^\psi_{x,z}(Z(y)),
	$$
	from which we conclude that $L^\psi_{x,Z(y)}(y)\leq L^\psi_{x,z}(y)$, and then combining with \eqref{eq:A}, $\psi(y) \leq L^\psi_{x,z}(y)$. This finishes the proof.
\end{proof}
\begin{proof}[Proof of \cref{lem:equal_hulls}]
	First, since $f\geq f^c$, $(f-g)^c\geq (f^c-g)^c$. On the other hand, we have $(f-g)^c\leq (f-g)$ and therefore $(f-g)^c+g\leq f$. Since the sum of two convex functions is convex, $(f-g)^c+g$ is also convex. Hence, $(f-g)^c+g\leq f^c$, and therefore $(f-g)^c\leq (f^c-g)$. Since $(f^c-g)^c$ is the largest convex function dominated by $(f^c-g)$, $(f-g)^c\leq (f^c-g)^c$. It follows that $(f-g)^c= (f^c-g)^c$.
\end{proof}

\begin{proof}[Proof of \cref{CSpotential}]
	
	Since $g\in\sD(a,b)$ and $f\in\sD(\alpha,\beta)$ with $g\geq f$, we have that
	\begin{align*}
	0\leq\lim_{k\to\infty}\{g(k)-f(k)\}&=\lim_{k\to\infty}\{g(k)-(ak-b)-f(k)+(\alpha k-\beta)+(a-\alpha)k-(b-\beta)\}\\
	&=\lim_{k\to\infty}\{(a-\alpha)k-(b-\beta)\}=\lim_{k\to\infty}h(k),
	\end{align*}
	and therefore $a\geq\alpha$. Also, if $\alpha = a$ then $\beta \geq b$.
	
	

	Now suppose $f \leq g$ and $g-f \geq h$.
	Then $g - f\geq h^+$ and since $h^+$ is convex, we have that $(g-f)\geq (g-f)^c\geq h^+$. Then, $\lim_{\lvert k\lvert\to\infty}\{g(k)-f(k)-  h^+(k) \}=0$, and it follows that $(g-f)^c\in\sD(a-\alpha,b-\beta)$.
\end{proof}

\bibliographystyle{plainnat}

\begin{thebibliography}{99}



\bibitem{Backhoff:19}
{\sc Backhoﬀ-Veraguas J., Beiglb\"{o}ck M., Pammer G.:} {\em Existence, duality, and cyclical monotonicity for weak transport costs.} Calc. Var., 58(203), (2019). 


\bibitem{BeiglbockCox:17}
{\sc Beiglb{\"o}ck, M., Cox A.M.G., Huesmann, M.:} {\em The geometry of multi-marginal Skorokhod Embedding.} Probab. Theory Relat. Fields, (2019). 

\bibitem{BeiglbockHenryLaborderePenkner:13}
{\sc Beiglb{\"o}ck M., Henry-Labord{\`e}re P., Penkner F.:} {\em Model-independent bounds for option prices\textemdash mass transport approach.} Finance Stoch., 17(3):477--501, (2013). 

\bibitem{BeiglbockHenryLabordereTouzi:17}
{\sc Beiglb{\"o}ck M., Henry-Labord{\`e}re P., Touzi N.:} {\em Monotone martingale transport plans and {S}korokhod embedding.} Stochastic Process. Appl., 127(9):3005--3013, (2017). 


\bibitem{BeiglbockJuillet:16}
{\sc Beiglb{\"o}ck M., Juillet N.:} {\em On a problem of optimal transport under marginal martingale constraints.} Ann. Probab., 44(1):42--106, (2016). 

\bibitem{BeiglbockJuillet:16s}
{\sc Beiglb{\"o}ck M., Juillet N.:} {\em Shadow couplings.} arXiv preprint, (2016). Available online at: https://arxiv.org/abs/1609.03340



\bibitem{Brenier:87}
{\sc Brenier Y.:} {\em D{\'{e}}composition polaire et r{\'{e}}arrangement monotone des champs de vecteurs.}
C. R. Acad. Sci. Paris S\'{e}r. I Math., 305(19): 805-808, (1987).

\bibitem{BHJ:20}
{\sc Br{\"u}ckerhoff M., Huesmann M., Juillet N.:} {\em Shadow martingales--a stochastic mass transport approach to the peacock problem.} arXiv preprint, (2020). Available online at: https://arxiv.org/abs/2006.10478


\bibitem{Campi:17}
{\sc Campi L., Laachir I., Martini C.:} {\em Change of numeraire in the two-marginals martingale transport problem.} Finance Stoch., 21:471--486, (2017). 


\bibitem{Chacon:77}
{\sc Chacon, R.V.:} {\em Potential processes.} Trans. Amer. Math. Soc., 226:39--58, (1977).

\bibitem{ChaconWalsh:76}
{\sc Chacon, R.V., Walsh J.B.:} {\em One-dimensional potential embedding.} In S{\'e}min. Probab. X, vol. 511 of Lecture Notes in Mathematics, 19--23, Springer, Berlin, (1976). 








\bibitem{GalichonHenryLabordereTouzi:14}
{\sc Galcihon A., Henry-Labord\`ere P., Touzi N.:} {\em A stochastic control approach to no-arbitrage bounds given marginals with an application to lookback options.} Ann. Appl. Probab., 24(1):313--336, (2014). 

\bibitem{Gozlan:18}
{\sc Gozlan N., Juillet N.:} {\em On a mixture of Brenier and Strassen Theorems.} Proc. Lond. Math. Soc., 120(3):434--463, (2020). 

\bibitem{Gozlan:17}
{\sc Gozlan N., Roberto C., Samson P.-M., Tetali P.:} {\em Kantorovich duality for general transport costs and applications.} J. Funct. Anal., 273(11):3327--3405, (2017). 



\bibitem{HenryLabordereTouzi:16}
{\sc Henry-Labord{\`e}re P., Touzi N.:} {\em An explicit martingale version of the one-dimensional Brenier's theorem.} Finance Stoch., 20(3):635--668, (2016).
 
 \bibitem{HenryLabordereTanTouzi:16}
 {\sc Henry-Labord{\`e}re P., Tan X., Touzi N.:} {\em An explicit martingale version of the one-dimensional Brenier’s Theorem with full marginals constraint.} Stochastic Process. Appl., 126(9):2800--2834, (2016).



\bibitem{peacock}
{\sc Hirsch F., Roynette B.:}  {\em A new proof of Kellerer's theorem.} ESAIM Probab. Stat., 16:48--60, (2012).




\bibitem{HobsonKlimmek:15}
{\sc Hobson D.G., Klimmek M.:} {\em Robust price bounds for the forward starting straddle.} Finance Stoch., 19(1):189--214, (2015). 

\bibitem{HobsonNeuberger:12}
{\sc Hobson D.G., Neuberger A.:} {\em Robust bounds for forward start options.} Math. Finance, 22(1):31--56, (2012). 




\bibitem{HobsonNorgilas:18}
{\sc Hobson D.G., Norgilas D.:} {\em The left-curtain martingale coupling in the presence of atoms.} Ann. Appl. Probab., 29(3):1904--1928, (2019). 


\bibitem{Juillet:16}
{\sc Juillet N.:} {\em Stability of the shadow projection and the left-curtain coupling.} Ann. Inst. Henri Poincar\'{e} Probab. Stat., 52(4):1823--1843, (2016). 

\bibitem{Juillet:18}
{\sc Juillet N.:} {\em Martingales associated to peacocks using the curtain coupling.} Electron. J. Probab., 23(9):1--29, (2018). 


\bibitem{Kell:72}
{\sc Kellerer H.G.:} {\em Markov-Komposition und eine Anwendung auf Martingale.} Math. Ann., 198:99--122, (1972). 


\bibitem{Kell:73}
{\sc Kellerer H.G.:} {\em Integraldarstellung von Dilationen.} {In Transactions of the Sixth Prague Conference on Information Theory, Statistical Decision Functions, Random Processes (Tech. Univ., Prague, 1971; dedicated to the memory of Anton\'{i}n \v{S}pa\v{c}ek)}, pp. 341--374. Academia, Prague, (1973).





\bibitem{NutzStebegg:18}
{\sc Nutz M., Stebegg F.:} {\em Canonical supermartingale couplings.} Ann. Probab., 46(6):3351--3398, (2018). 

\bibitem{NutzStebeggTan:17}
{\sc Nutz M., Stebegg F., Tan X.:} {\em Multiperiod martingale transport.} Stochastic Process. Appl., 130(3):1568-1615, (2020).  


\bibitem{Rockafellar:72}
{\sc Rockafellar R.~T.:} {\em Convex analysis.} (No. 28) Princeton Univ. Press, (1970).

\bibitem{RuRach:90}
{\sc R\"{u}schendorf L., Rachev S.T.:} {\em A characterization of random variables with minimum $L^2$-distance.} J. Multivariate Anal., 32:48--54, (1990). 




\bibitem{Strassen:65}
{\sc Strassen V.:} {\em The existence of probability measures with given marginals.} Ann. Math. Statist., 36(2):423-439, (1965).



\end{thebibliography}

\end{document}